\documentclass[10pt]{amsart}
\usepackage[utf8]{inputenc}
\usepackage[T1]{fontenc}
\usepackage[english]{babel}
\title[Hausdorff dimension of the harmonic measure for rel. hyp. groups]{The Hausdorff dimension of the harmonic measure for relatively hyperbolic groups}
\author{Matthieu Dussaule and Wenyuan Yang}

\date{}
\usepackage{silence}
\usepackage{soul}
\usepackage{amsmath}
\usepackage{amssymb}
\usepackage{dsfont}
\usepackage{amsfonts}
\usepackage{amsthm}
\usepackage{tikz}
\usepackage{graphicx}
\usepackage{fancyhdr}
\usepackage{caption}
\usepackage{enumerate}
\usepackage{mathtools}
\usepackage[colorlinks=true,pdfstartview=FitH,bookmarks=false]{hyperref}

\hypersetup{urlcolor=black,linkcolor=black,citecolor=black,colorlinks=true}

\usepackage{geometry} 
\usepackage{ifthen}

\newcommand{\showcomments}{yes}

\newsavebox{\commentbox}
{\ifthenelse{\equal{\showcomments}{yes}}%
{\footnotemark
    \begin{lrbox}{\commentbox}
    \begin{minipage}[t]{1.25in}\raggedright\sffamily\upshape\tiny
    \footnotemark[\arabic{footnote}]}%
{\begin{lrbox}{\commentbox}}}%
{\ifthenelse{\equal{\showcomments}{yes}}%
{\end{minipage}\end{lrbox}\marginpar{\usebox{\commentbox}}}%
{\end{lrbox}}}

\theoremstyle{definition}

\theoremstyle{plain}
\newtheorem{definition}{Definition}[section]
\newtheorem{prop}[definition]{Proposition}
\newtheorem{cor}[definition]{Corollary}
\newtheorem{theorem}[definition]{Theorem}
\newtheorem{lem}[definition]{Lemma}
\newtheorem*{thm*}{Theorem}
\newtheorem*{prop*}{Proposition}
\newtheorem*{lem*}{Lemma}
\newtheorem{claim}[definition]{Claim}

\theoremstyle{remark}
\newtheorem{remark}[definition]{Remark}

\newtheorem*{rem*}{Remark}

\DeclareMathOperator{\Tr}{Tr}
\DeclareMathOperator{\diam}{diam}
\DeclareMathOperator{\Bn}{Bn}

\usepackage{etoolbox}
\apptocmd{\sloppy}{\hbadness 10000\relax}{}{}
\apptocmd{\sloppy}{\vbadness 10000\relax}{}{}

\begin{document}

\maketitle

\begin{abstract}
The paper  studies the Hausdorff dimension of harmonic measures on various boundaries of a relatively hyperbolic group which are associated with random walks driven by a probability measure with finite first moment.
With respect to the Floyd metric and the shortcut metric, we prove that the Hausdorff dimension of the harmonic measure equals the ratio of the entropy and the drift of the random walk. 

If the group is infinitely-ended,  the same dimension formula is obtained for the end boundary endowed with a visual metric.
In addition, the Hausdorff dimension of  the visual metric is identified with the growth rate of the word metric.
These results are complemented by a characterization of doubling visual metrics for accessible infinitely-ended groups : the visual metrics on the end boundary is doubling if and only if the group is virtually free.
Consequently, there are at least two different   bi-H\"older classes (and thus quasi-symmetric classes) of visual metrics on the end boundary.
\end{abstract}

\section{Introduction}
Relatively hyperbolic groups can admit several interesting compactifications such as the Floyd \cite{Floyd}, Bowditch \cite{Bowditch} and Freudenthal \cite{Freudenthal} (also called end) compactification. The resulting boundaries are compact metrizable spaces on which the groups act  and  have rich dynamics in terms of convergence actions. This point of view has found many applications \cite{Bow98} \cite{GePoJEMS}.
In addition, these boundaries can be endowed with two well-known classes of (quasi-)conformal \cite{Coornaert} and harmonic measures \cite{KarlssonPoisson} so they are examples of metric measured spaces $(X, d, \nu)$.
The Hausdorff dimension $Hdim_d(\nu)$ of the triple $(X, d, \nu)$ is the infimum of the Hausdorff dimensions of $\nu$-full subsets, so it is a measurement of the largeness of the measure class of $\nu$.
Comparison of the conformal and the harmonic measures has been an active research problem with origins in dynamic systems, see \cite{KaimanovichVershik} \cite{GuiLeJan} \cite{BHM} \cite{GMM} \cite{GekhtmanDussaule} \cite{GGPY} \cite{RanTiozzo} to just name a few.

This paper is devoted to computing the Hausdorff dimension formula for harmonic measures on various boundaries associated with a random walk with finite first moment on a non-elementary relatively hyperbolic group.
Precisely, we compute the Hausdorff dimension of the harmonic measure $\nu$ on the Floyd and Bowditch boundaries endowed respectively with the Floyd and the Floyd shortcut distances.
We also compute this Hausdorff dimension on the end boundary endowed with a visual distance whenever the group is infinitely ended.
Up to a parameter depending only on the chosen distance, we show that
\begin{equation}\label{HdimFormulaEQ}
Hdim(\nu) = \frac{h}{l}    
\end{equation}
where $h$ is the asymptotic entropy and $l$ is the rate of escape in the word metric of the random walk.
See Theorem \ref{mainthm1} and Theorem \ref{mainthm2} for accurate statements.

This formula (\ref{HdimFormulaEQ}) was first obtained by Kaimanovich \cite{KaimanovichHD} and Ledrappier \cite{Ledrappier} on free groups and then by Le Prince \cite{LePrincelowerbound} on general hyperbolic groups. A substantial generalization of Le Prince's results was given by Tanaka \cite{Tanaka} to any proper action with exponential growth on proper hyperbolic spaces with bounded geometry, and any acylindrical action on possibly improper hyperbolic spaces. In particular, his results do apply to both the geometrical finite action  of a relatively hyperbolic group  on a proper hyperbolic space,  and the acylindrical action  on its relative Cayley graph \cite{Osin16}.

However,  the former proper action under bounded geometry assumption forces peripheral subgroups to be virtually nilpotent \cite{DY05}. This makes a serious limitation of Tanaka's result to be applied in this rather general class of groups. Moreover, visual distances on the Bowditch boundary are very non-canonical, for they can depend heavily on the choice of a hyperbolic space $X$ on which the group acts, see \cite{Healy}.
On the other hand, the acylindrical action on the relative Cayley graph does yield the formula~(\ref{HdimFormulaEQ}) for the harmonic measure on the Gromov boundary of the relative Cayley graph and where the drift is computed in the relative metric. In a certain point of view, this is unsatisfactory since the relative metric on the group is non-proper and the (non-compact) Gromov boundary here is only a part of the compact Bowditch boundary.

One of our main contributions here is to consider the drift for the word metric and to compute the Hausdorff dimension associated to the Floyd shortcut distance on the Bowditch boundary, which only depends on the group and the choice of a word distance.



\subsection{Hausdorff dimension of harmonic measures}
Let $\mu$ be a probability measure on a finitely generated group $G$ such that the support $\mathrm{supp}(\mu)$ generates $G$ as a semi-group.
We call such a measure \textit{admissible}.
The measure defines a $\mu$-random walk with step transitions given by $p(x,y)=\mu(x^{-1}y)$ for $x,y\in G$.
Let $$\Omega:=\{\mathbf x=(\omega_n)_{n\ge 0}: \omega_n\in G\}$$ be the trajectory space of the $\mu$-random walk with the probability measure $\mathbb P$, which  is the pushforward of the product measure $(G^{\mathbb N}, \mu^{\mathbb N})$ under the product map $G^{\mathbb N} \to\Omega$ given by
$$
(s_1,s_2, \cdots, s_n,\cdots) \mapsto (1, s_1, s_1s_2,\cdots, w_n,\cdots),
$$ 
where $\omega_n=s_1\cdots s_n$.

Denote by
$$L(\mu)=\sum_{g\in G}d(1,g)\mu(g)$$ the expectation of $d(1,\omega_1)$ and more generally by
$$L(\mu^{*n})=\sum_{g\in G}d(1,g)\mu^{*n}(g)$$ the expectation of $d(1,\omega_n)$
where $\mu^{*n}$ is the $n$-th convolution power of $\mu$.
Equivalently, $\mu^{* n}$ is the law of the random variable $\mathbf x\mapsto \omega_n$.
The sequence $L(\mu^{*n})$ is sub-additive and so $\frac{1}{n}L(\mu^{*n})$ has a well defined limit $l$ called the \textit{rate of escape} (or  \textit{drift}).
Whenever $\mu$ has \textit{finite first moment} (i.e. $L(\mu)<\infty$), Kingman's sub-additive ergodic theorem shows that almost surely,
\begin{equation}\label{equationdefl}
l:=\lim_{n\to\infty} \frac{d(1,\omega_n)}{n}<+\infty.
\end{equation}
Also define  $H(\mu)=\sum_{g\in G} \mu(g) \log \mu(g)$ and $H(\mu^{*n})=\sum_{g\in G} \mu^{*n}(g) \log \mu^{*n}(g)$.
The sequence $H(\mu^{*n})$ also is sub-additive and so $\frac{1}{n}H(\mu^{*n})$ has a well defined limit $h$ called the asymptotic \textit{entropy}.
Again, whenever $H(\mu)$ is finite, an application of the ergodic Theorem (see \cite{Der80} or \cite{KaimanovichVershik}) shows that
\begin{equation}\label{equationdefh}
h:=\lim_{n\to\infty} \frac{-\log (\mu^{*n}(\omega_n))}{n}<+\infty.
\end{equation}

The groups $G$ under consideration are assumed to be relatively hyperbolic throughout.
There are many equivalent ways to formulate this notion.
To motivate our results, we use the dynamical definition. An action of $G$ by homeomorphism on a compact metrizable space $M$ is called \textit{convergence} if the induced action on the space of triple points is properly discontinuously. Let $(\Gamma, d)$ denote the Cayley graph of $G$ with respect to a finite generating set. Then $G$ is called \textit{relatively hyperbolic}   if  there exists a Hausdorff compact  space $M$ compactifying every Cayley graph $\Gamma$   so that the left multiplication of $G$ extends to  a minimal geometrically finite action on the boundary $M$. See precise definitions in Section~\ref{SSRHG}.  The compact space $M$ denoted by $\partial_{\mathcal B} G$ later on is called \textit{Bowditch boundary} of $G$.
{A relatively hyperbolic group $G$ is called \textit{non-elementary} if its Bowditch boundary contains more than two points.
Equivalently, $G$ fixes no finite subset of $\partial_{\mathcal{B}}G$.
In such a case, $G$ is non-amenable.}
In particular, the rate of escape and the asymptotic entropy are positive.

There are two natural classes of metrics on the Bowditch boundary.
First, by Yaman \cite{Yaman}, $\partial_{\mathcal B} G$ can be realized as the Gromov boundary of a proper hyperbolic space $X$ on which $G$ acts via a geometrically finite action. Thus, we can endow $\partial_{\mathcal B} G$ with the visual metric constructed using the hyperbolic geometry of $X$.
Second, the Floyd (shortcut) metric is obtained from the Cayley graph $(\Gamma, d)$ as follows.
We fix a parameter $\lambda\in (0,1)$ and a basepoint $o\in G$. Rescaling the length of every edge $e$ of $\Gamma$ to $\lambda^{d(o, e)}$ induces a new length metric called the \textit{Floyd metric} $\delta_{\lambda}$ on $\Gamma$. The Cauchy metric completion of $(\Gamma,\delta_{\lambda})$ is the so-called Floyd compactification of $\Gamma$ on which $G$ acts as a convergence action \cite{Karlssonfreesubgroups}.
By the work of Gerasimov and Potyagailo \cite{Gerasimov}, \cite{GePoJEMS}, there exists $\lambda_0$ such that whenever $\lambda\in [\lambda_0,1)$, the Bowditch boundary is an equivariant quotient of the Floyd boundary where the non-trivial fibers are possible only on bounded parabolic points.
The Floyd metric can thus be pushforwarded to obtain the so-called shortcut metric $\bar{\delta}_{\lambda}$ on  $\partial_{\mathcal B} G$ \cite{GePoJEMS}. 
See more details in Section~\ref{SSFloyd}.

In \cite{KarlssonPoisson}, Karlsson proved that almost every trajectory converges to a limit point in the Floyd boundary, and the \textit{hitting measure} or \textit{harmonic measure} $\nu_{\mathcal{F}}$ on the Floyd boundary gives  a model of the Poisson boundary of the $\mu$-random walk \cite{Kaimanovichboundaries}. The same discussion applies to the Bowditch boundary which gives another model of the Poisson boundary when endowed with a harmonic measure denoted by $\nu_{\mathcal{B}}$.
We are now ready to state the first main result.

\begin{theorem}\label{mainthm1}[Theorem~\ref{theoremHausdorffharmonicFloydBowditch}]
Suppose $G$ is a non-elementary relatively hyperbolic group and fix a finite generating set for $G$.
Also suppose that $\mu$ is an admissible probability measure with finite first moment on $G$. Then there exists $\lambda_0\in (0, 1)$ such that  for every $\lambda\in [\lambda_0, 1)$, 
$$
Hdim_{{\delta}_{\lambda}}(\nu_{\mathcal{F}})=Hdim_{\bar{\delta}_{\lambda}}(\nu_{\mathcal{B}})=\frac{-1}{\log \lambda} \frac{h}{l}.
$$
Moreover, the measure $\nu_{\mathcal{B}}$ on $(\partial_{\mathcal B} G, \bar{\delta}_\lambda)$ and the measure $\nu_{\mathcal{F}}$ on   $(\partial_{\mathcal F} G, {\delta}_\lambda)$ are exact dimensional. 
\end{theorem}

Denote by $S_n=\{g\in G: d(1, g)=n\}$ the sphere of radius $n$ in the Cayley graph $(\Gamma, d)$. We define the \textit{growth rate} of (the Cayley graph of) $G$ as
$$
v:=\lim_{n\to\infty} \frac{\log \sharp S_n}{n}
$$
The two quantities $l$ and $v$ both depending on the word metric of $\Gamma$ are related to the entropy $h$ which only depends on the measure $\mu$ by the following fundamental inequality
$$
\frac{h}{l} \le v$$
also called the Guivarc'h inequality \cite{Derriennic}.
This inequality holds for any $\mu$-random walk.

In \cite{PY}, it is proved that $\frac{-v}{\log \lambda}$ is the Hausdorff dimension of the Floyd and Bowditch boundaries. 
The following corollary thus follows from  the strictness of the fundamental inequality  established  in \cite[Theorem~1.3, Theorem~1.6]{GekhtmanDussaule} for certain relatively hyperbolic groups. Recall that $\mu$ has \textit{finite super-exponential moment} if  $\sum_{g\in G} \exp(cd(1,g)) \mu(g)$ is finite for every $c>0$.
\begin{cor}
Suppose $G$ is a non-elementary relatively hyperbolic group and $\mu$ is an admissible probability measure on $G$ with finite super-exponential moment.
If one of the parabolic subgroups is virtually abelian of rank at least 2, or if the Bowditch boundary is homeomorphic to a sphere of dimension at least 2, then
$$
Hdim_{\bar{\delta}_{\lambda}}(\nu_{\mathcal{B}}) < Hdim_{\bar{\delta}_{\lambda}}(\partial_{\mathcal{B}} G).
$$
\end{cor}

\subsection{Hausdorff dimension of the end boundary and of harmonic measures with respect to visual metrics}
According to a celebrated result of Stallings \cite{Stallingstorsionfree}, \cite{Stallings}, any infinitely ended group $G$ splits nontrivally as an amalgamated product $A*_CB$ or an HNN extension $A*_C$, where $C$ is a finite group.
The action on the corresponding Bass-Serre tree satisfies the conditions of \cite[Definition~2]{Bowditch} and so $G$ is relatively hyperbolic.
Moreover, in the former case, the maximal parabolic subgroups are exactly the conjugates of $A$ and the conjugates of $B$.
In the later case, they are exactly the conjugates of $A$.
In addition to the Floyd and Bowditch boundary, $G$ can be compactified using the Freudenthal (or end) boundary introduced by Freudenthal \cite{Freudenthal}.
The interplay between the end boundary and asymptotic properties of random walks is a well-studied subject, see for instance \cite{PW1}, \cite{PW2}, \cite{Woessends}.

Let $\partial_{\mathcal{E}} G$ be the end boundary of an infinitely  ended group $G$. The topology of the end boundary is independent of the choice of Cayley graph.
For every $\lambda\in (0,1)$,  we define a \textit{visual metric}, extending the definition of Candellero, Gilch and M\"uller \cite{CandelleroGilchMuller} for free products.
This metric was  independently studied by Cornulier in \cite{Cor19}.
Fixing a generating set, let $\Gamma$ be the Cayley graph of $G$ and let ${\rho}_{\lambda}$ be the visual metric defined on the end compactification of $\partial_{\mathcal{E}} G\cup \Gamma$. Precisely, define the distance $\rho_\lambda(\xi, \eta)=\lambda^n$ between two ends $\xi, \eta$ if $n$ is the minimal radius of the ball at a basepoint separating $\xi$ and $\eta$.
A quasi-isometry $f$ induces a homeomorphism $\tilde{f}$ between end boundaries.
Moreover, this induced homeomorphism is \textit{bi-H\"older}: for some $\alpha, \beta>0,c>1$ and for all pair  $(\xi,\eta)$ of points,  
$$
 c^{-1} \rho_\lambda(\xi, \eta)^\alpha \le \rho_\lambda(\tilde{f}(\xi),\tilde{f}(\eta)) \le c \rho_\lambda(\xi, \eta)^\beta
$$
which holds more generally for sublinearly bi-Lipschitz equivalence by \cite[Corollary 1.4]{Cor19ENS}. 

Our second main result is analogous  to Theorem \ref{mainthm1} and computes the Hausdorff dimension of the end boundary  and of  harmonic measures with respect to visual metrics.
Let us denote by $\nu_{\mathcal{E}}$ the harmonic measure on the end boundary and $h, l$ the entropy and drift of the $\mu$-random walk. 
Also, let us denote by $v$ the growth rate of $G$.

\begin{theorem}\label{mainthm2}[Theorem~\ref{theoremHausdorffharmonicends}, Theorem~\ref{Hausdorffdimensionsetofends}]
Suppose $G$ is a finitely generated group with infinitely many ends.
Also suppose that $\mu$ is an admissible probability measure with finite first moment on $G$. Then, for every Cayley graph $\Gamma$ and every $\lambda\in (0,1)$, 
$$Hdim_{{\rho}_{\lambda}}(\nu_{\mathcal{E}})=\frac{-1}{\log \lambda} \frac{h}{l}$$
and
$$Hdim_{{\rho}_{\lambda}}(\partial_{\mathcal{E}} G)=\frac{-v}{\log \lambda}.$$
\end{theorem}
\begin{remark}
The second equality was proved for free products with a standard generating set in \cite{CandelleroGilchMuller}. Our proof is very different and applies to any infinitely ended group with any finite generating set.
\end{remark}

The archetypal groups with infinitely many ends are free products of the form $A*B$, where $A\neq \mathbb{Z}/2\mathbb{Z}$ or $B\neq \mathbb{Z}/2\mathbb{Z}$.
Using again the strictness of the fundamental inequality  established  in \cite[Theorem~1.5]{GekhtmanDussaule} for certain free products, we can state the following corollary.
\begin{cor}
Suppose $G$ is a free product $A*B$ and assume that $A$ is a non-virtually cyclic nilpotent group.
Let $\mu$ be a probability measure on $G$ with finite super-exponential moment.
Then,
$$
Hdim_{\rho_{\lambda}}(\nu_{\mathcal{E}}) < Hdim_{\rho_{\lambda}}(\partial_{\mathcal{E}} G).
$$
\end{cor}

\subsection{Doubling property of the end boundary}
The study of the harmonic measure associated with a random walk is closely related to the doubling property of the involved boundary,
see \cite[Section~4]{Tanaka}.
We also refer to \cite[Proposition 4.12]{Haissinsky}, where the doubling property is combined with the so-called shadow lemma to give a short proof of the formula $Hdim(\nu)=h/l$ in the setting of hyperbolic groups.
It is not known whether the Floyd boundary equipped with the Floyd distance and the Bowditch boundary equipped with the shortcut Floyd distance are doubling, see \cite[Question 1.7]{PY}.
On the other hand, to give a more complete picture, we clarify the situation for the end boundary equipped with a visual distance.

A metric space $(X, d)$ is called \textit{doubling} if there exists a constant $N>0$ such that every ball of radius $\delta>0$ can be covered by at most $N$ balls of radius $\delta/2$.
Equivalently, for any $\theta\in(0,1)$, there exists $N(\theta)>0$ such that every ball of radius $\delta$ can be covered by at most $N(\theta)$ balls of radius $\theta \delta$.
The doubling property is known to be a bi-H\"older invariant.
A measure $\mu$ on $(X, d)$ is called \textit{doubling} if there exists a constant $C$ such that $\mu(2B)\le C\mu(B)$ for every ball $B$, where $2B$ denotes the ball with the same center and a radius twice as large as $B$.
The existence of a doubling measure on a metric space $(X,d)$ implies that the metric space  $(X,d)$ is doubling. The converse is  true for complete metric spaces.

It is well-known that the Patterson-Sullivan measure on the Gromov boundary of a hyperbolic group is Alhfors regular, hence doubling, for the Gromov's visual metric, so the Hausdorff dimension of the Patterson-Sullivan measure equals that of the whole boundary, see \cite{Coornaert}.
One can construct, as in \cite{YangPS}, a class of Patterson-Sullivan measures on the end boundary through the action on the Cayley graph.
Those measures yield quasi-conformal densities without atoms.
This motivates the question whether the Patterson-Sullivan measure is Alhfors regular or at least doubling on the end boundary endowed with a visual metric.

Our third main result gives a characterization of the doubling property of the end boundary for \textit{accessible} infinitely-ended groups.
Such groups admit a splitting as a graph of groups over finite edge groups so that the vertex groups are either finite or one-ended.
It is a famous result of Dunwoody \cite{Dun85} that finitely presented group are accessible.

\begin{theorem}\label{CharDoublingThm}
Let $G$ be a finitely generated,   accessible,    infinitely-ended group.
Let $\lambda\in (0,1)$ and endow the end boundary $\partial_{\mathcal E}G$ of $G$ with the visual metric $\rho_\lambda$.
Then, $(\partial_{\mathcal{E}}G,\rho_\lambda)$ is doubling if and only if $G$ is virtually free.
\end{theorem}

From the above discussion, the following two corollaries are immediate.

\begin{cor}
The Patterson-Sullivan measure on the end boundary of an accessible infinitely-ended group is doubling for the visual metric if and only if the group is virtually free.
\end{cor}

The following one addresses an analogous question \cite[Question 4.2]{Tanaka} in our setup.

\begin{cor}
If  an accessible infinitely-ended group is not virtually free, then the harmonic measure $\nu_{\mathcal E}$ on the end boundary endowed with a visual metric is not doubling.
\end{cor}

A result of David-Semmes \cite[Theorem 15.5]{Hei01} says that a metric space is quasi-symmetric to the standard Cantor ternary set if and only if it is compact, doubling, uniformly perfect and uniformly disconnected. Precise definitions of the latter two can be found in \textsection 11.1 and \textsection 14.24 in \cite{Hei01}. It is known that a doubling space is uniformly perfect and a ultrametric space is uniformly disconnected. The end boundary of a virtually free group with visual metric is doubling by \cite{Coornaert} and thus quasi-symmetric to the standard Cantor ternary set. 

By \cite[Corollary 11.3]{Hei01}, quasi-symmetric maps between uniformly perfect compact spaces are bi-H\"older.
However,  the converse is not true: bi-H\"older  homeomorphims are  not necessarily  quasi-symmetric.
Interesting examples are given by the boundaries  of the real hyperbolic space $\mathbb H^4$ and the complex hyperbolic space $\mathbb {CH}^2$ equipped with Gromov's visual metrics, which are bi-H\"older but not quasi-symmetric, since they can be distinguished by their conformal dimension.
Thus, a bi-H\"older classification of visual metrics on the end boundary is a reasonable and interesting problem. We refer the reader to \cite[Introduction 1D]{Cor19ENS} for further discussion.
Recalling that the doubling property is a bi-H\"older invariant, we obtain the following corollary, which answers positively a question of  Cornulier \cite[Question 1.27]{Cor19}.  
\begin{cor}\label{CornulierQCor}
The end boundary of an accessible infinitely-ended group is  bi-H\"older equivalent to the standard Cantor ternary set if and only if it is virtually free. 

In particular, the end boundaries of a free group and the free product of two 1-ended groups    are not  bi-H\"older equivalent.
\end{cor}

Inaccessible groups do exist by work of Dunwoody \cite{Dun93}.
The proof of Theorem \ref{CharDoublingThm} fails generally for inaccessible groups,
but we can still give a result for such groups under some additional assumptions, see precisely Proposition~\ref{propinaccessible}.

Finally, let us compare the visual metric on the end boundary with the Gromov's visual metric coming from an action of $G$ on a hyperbolic space $X$.
Fix a splitting of an infinitely-ended group $G$ over finite   groups as a finite graph of groups.
As explained above, by \cite[Definition~2]{Bowditch}, this splitting makes $G$  hyperbolic relative to the set of vertex groups.
Hence, by Yaman \cite{Yaman}, $G$ acts via a geometrically finite action on a proper hyperbolic space $X$ so that the Bowditch boundary is homeomorphic to the Gromov boundary of $X$.
If $G$ is accessible and the splitting is terminal, then the end boundary is equivariantly homeomorphic to the Bowditch boundary, see Section~\ref{SSRelHyp} for more details.
We can thus endow the end boundary with the Gromov's visual metric coming from the hyperbolic space $X$ on which $G$ acts.
According to \cite{DY05} and \cite{BS00}, { $X$ can be chosen so that}
this metric is doubling if and only if the parabolic subgroups are virtually nilpotent.
{In such a situation,} by Theorem \ref{CharDoublingThm} the two possible metrics on the end boundary cannot be in the same bi-H\"older class unless the group is virtually free, for one is doubling and the other is not.
\begin{cor}
Assume that $G$ is not virtually free and splits over finite groups as a finite graph of virtually nilpotent groups. Then, there exists a hyperbolic space $X$ on which $G$ acts via a geometrically finite action such that the visual metric on the end boundary is not bi-H\"older equivalent to the  Gromov's visual metric coming from this action.
\end{cor}

\begin{remark}
Similarly as in  Corollary~\ref{CornulierQCor}, we can derive the following result from \cite{DY05}.
Let  $G_1=H_1*\mathbb{Z}$ and $G_2=H_2*\mathbb{Z}$, where $H_1$ is virtually nilpotent and $H_2$ is one-ended but not virtually nilpotent. Then $G_1$ and $G_2$ admit  geometrically finite actions on proper hyperbolic spaces $X_1$ and $X_2$ with bounded geometry so that their Gromov boundaries endowed with Gromov's visual metric  are homeomorphic to the Cantor sets but are not bi-H\"older equivalent, for one boundary is doubling and not the other.
\end{remark}

\subsection*{Overview and organisation of the paper}
In Section~\ref{PrelimSection}, we review all the preliminary results we will need in the following.
We recall the definition of relatively hyperbolic groups, the Bowditch boundary, the Floyd distance and the Floyd boundary.
We also give more details about the Hausdorff dimension of a finite measure on a metric space.

Section~\ref{FloyBowditchSection} is devoted to the proof of Theorem~\ref{mainthm1}, which treats separately the upper bound (Proposition~\ref{upper-boundharmonicFloyd}) and lower bound  (Proposition~\ref{lower-boundharmonicFloyd}) for the Hausdorff dimension.
The lower bound follows a strategy similar to the one developed by Tanaka \cite{Tanaka}. One of the main tools in \cite{Tanaka} is that the random walk sub-linearly tracks geodesics $[1,\omega_\infty]$ on the hyperbolic space $X$ on which the group $G$ acts, where $\omega_\infty$ is the limit of the random walk in the Gromov boundary of $X$.

In our situation,   the Cayley graph of a relatively hyperbolic group $G$  is generally not hyperbolic anymore, so it cannot be expected that the random walk stay close to any point on a geodesic $[1,\omega_\infty]$ in the Cayley graph. However, using Maher-Tiozzo \cite{MaherTiozzo} and Tiozzo \cite{Tiozzo}, we can prove that the random walk sub-linearly tracks word geodesics along transition points,
which are the points that are not deep in parabolic subgroups, see Definition~\ref{deftransition}.
Along the way, we also prove that the random walk spends at most sub-linear time in parabolic subgroups.
Precisely, we prove the following result.
For a given word geodesic $\alpha$, let $\Tr \alpha$ be the set of transition points on $\alpha$.

\begin{theorem}\label{mainthmdeviation}[Proposition~\ref{propsublineardeviationtransitionpoints}, Corollary~\ref{coroparabolicdrift}]
Suppose $G$ is a non-elementary relatively hyperbolic group and fix a finite generating set for $G$.
Also suppose that $\mu$ is an admissible probability measure with finite first moment on $G$. Then,
\begin{equation}\label{equationthmsulineardeviation1}
\mathbb{P}\left (\sup d (\omega_n,\Tr \alpha)=o(n)\right )=1,
\end{equation}
where the supremum is taken over all geodesics $\alpha$ from $1$ to the limit $\omega_{\infty}$ of the random walk in the Bowditch boundary of $G$.
Moreover,
\begin{equation}\label{equationthmsulineardeviation2}
\mathbb{P}\left (\sup_{U\in \mathcal{P}}d_U(1,\omega_n)=o(n)\right )=1,
\end{equation}
where $\mathcal{P}$ is the set of all left cosets of a chosen full family of conjugacy classes of parabolic subgroups and where $d_U(x,y)$ is the distance between the projections of $x$ and $y$ on $U$.
\end{theorem}

This last result is a weak version (under a finite moment condition) of the results in \cite{SistoTaylor} for finitely supported random walks.
Indeed, it is proved in \cite[Theorem~2.3]{SistoTaylor} that such a random walk spends at most logarithmic time in parabolic subgroups.

With the sublinear tracking of transition points at hand, we use the estimates of Floyd disks by shadows in \cite{PY} to obtain the lower bound of Hausdorff dimension.  

In Sections~\ref{EndsSection} and~\ref{HDimEndsSection}, we prove Theorem~\ref{mainthm2}.
We first extend the definition of visual metrics introduced in \cite{CandelleroGilchMuller} to any group with infinitely many ends and we show that
$$Hdim_{{\rho}_{\lambda}}(\nu_{\mathcal{E}})=\frac{-1}{\log \lambda} \frac{h}{l}$$ in Section~\ref{EndsSection}.
Our proof is again similar to the proof in \cite{Tanaka}.
We introduce the notion of bottleneck : a set $V$ is a bottleneck between two points $x$ and $y$ if any path from $x$ to $y$ has to pass through a fixed neighborhood of $V$.
We then replace the sub-linear tracking of transition points by the sub-linear tracking of bottlenecks, see precisely Proposition~\ref{propsublineardeviationbottlenecks}.

We then prove that $$Hdim_{{\rho}_{\lambda}}(\partial_{\mathcal{E}} G)=\frac{-v}{\log \lambda}$$ in Section~\ref{HDimEndsSection}.
The proof follows the outline of an analogous result for the Bowditch and Floyd boundaries endowed with the Floyd (shortcut) distance \cite{PY}.
It is well-known that there exists a continuous and surjective map from $\partial_{\mathcal{F}} G$ to $\partial_{\mathcal{E}} G$ (see  \cite{Karlssonfreesubgroups} and \cite{GGPY}).
We further observe that the Floyd metric  dominates the visual metric through the map.
This gives the desired upper bound of $\partial_{\mathcal{E}} G$ by the same bound in \cite{PY} on the Hausdorff dimension of $\partial_{\mathcal{F}} G$.
Moreover, inspired by transition points, we use an enhanced version of the notion of bottlenecks introduced above to construct a sequence of  free semi-subgroups of $G$ whose set of ends has Hausdorff dimension arbitrarily close to $\frac{-v}{\log \lambda}$.
This proves the lower bound of $\partial_{\mathcal{E}} G$. 
A technical result in its proof is Proposition \ref{BiLipembedPop} saying that endowed with visual metric of corrected parameter depending on $\lambda$, the end boundary $\partial_{\mathcal E} T$ of free semi-groups is bi-Lipschitz embedded into $\partial_{\mathcal{E}} G$ with visual metric $\rho_\lambda$.
This greatly simplifies  the arguments in \cite[Section 3]{PY} using Patterson-Sullivan measure to estimate the Hausdorff dimension of $\partial_{\mathcal E} T$.     

Finally, Section~\ref{Sectiondoubling} deals with Theorem~\ref{CharDoublingThm}.
We consider an accessible infinitely ended group $G$.
Then, $G$ admits a splitting over finite edge groups as a finite graph of groups $\mathcal{G}$, so that the vertex groups either are finite or one-ended.
If every vertex group is finite, then $G$ is virtually free so that we can assume that one of the vertex group is one-ended.
Denote by $H$ such a vertex group.
The unique end $\xi$ of $H$ embeds into the end boundary of $G$.
We then show that for some fixed $\theta$, the ball of radius $\lambda^n$ centered at $\xi$ cannot be covered by $N(n)$ balls of radius $\theta \lambda^n$, where $N(n)$ goes to infinity, as $n$ tends to infinity, which concludes the proof.

\subsection*{Acknowledgments}
W.Y. is grateful to Giulio Tiozzo for explaining the results of \cite{Tiozzo} to him and also thanks Yves Cornulier for many helpful comments and corrections about the H\"older structure of the end boundary.
W. Y. is supported by the National Natural Science Foundation of China (No. 11771022).

\section{Preliminaries}\label{PrelimSection}
\subsection{Relatively hyperbolic groups}\label{SSRHG}
We now properly define relatively hyperbolic groups and recall several tools and results that will be used in the paper.
Let $G$ be a finitely generated group.
The action of $G$ on a compact Hausdorff space $T$ is called a \textit{convergence action} if the induced action on triples of distinct points of $T$ is properly discontinuous. Since $G$ is countable,  $T$ must be metrizable by \cite[Main Theorem]{Ge09}. Equivalently, the action $G\curvearrowright T$ is convergence if and only if every sequence of distinct elements $g_n$ in $G$ contains a subsequence $g_{n_k}$ such that $g_{n_k}\cdot x\to a$ and for all $x\in X$ with at most perhaps one exceptional point.

The set of accumulation points $\Lambda G$ of any orbit $G \cdot x\ (x\in T)$ is called the {\it limit set} of the action.
As long as $\Lambda G$ has more than two points, it is uncountable and it is then the unique minimal closed $G$-invariant subset of $T$.
The action is then said to be {\it non-elementary}.
In this case, the orbit of every point in $\Lambda G$ is infinite.
The action is {\it minimal} if $\Lambda G=T$.

A point $\zeta\in\Lambda G$  is called {\it conical} if there is a sequence $g_{n}$ of $G$ and distinct points $\alpha,\beta \in \Lambda G$ such that
$g_{n}\cdot \zeta \to \alpha$ and $g_{n}\cdot \eta \to \beta$ for all $\eta \in  T \setminus\{\zeta\}.$
The point $\zeta\in\Lambda G$ is called {\it bounded parabolic} if it is the unique fixed point of its stabilizer in $G$, which is infinite and acts cocompactly on $\Lambda G\setminus \{\zeta\}$.
The stabilizers of bounded parabolic points are called \textit{maximal parabolic subgroups}.
The convergence action $G\curvearrowright  T$ is called {\it geometrically finite} if every point of $\Lambda G\subset T$ is either conical or bounded parabolic.

\begin{definition}
Let $\mathcal{P}$ be a collection of subgroups of $G$. We say that $G$ is hyperbolic relative to $\mathcal{P}$ if there exists some compact Hausdorff space $T$ on which $G$ acts minimally and geometrically finitely and such that the maximal parabolic subgroups are exactly the elements of $\mathcal{P}$.
\end{definition}

In this situation, Yaman \cite{Yaman} proved that there exists a proper geodesic hyperbolic space $X$ on which $G$ acts such that the Gromov boundary of $X$ equivariantly coincides with $T$.
Further, Bowditch \cite{Bowditch} proved that the Gromov boundary of such a space $X$ is unique up to homeomorphism, hence so is $T$.
We call $T$ the {\it Bowditch boundary} of $G$ and we will denote it by $\partial_{\mathcal{B}}G$ in the following.
The union $G\cup \partial_{\mathcal{B}} G$ is called the Bowditch compactification.

\medskip
{Following Osin \cite{Osin}, we define the relative Cayley graph as follows.
We start with the Cayley graph $\Gamma$ associated with a finite generating set $S$.
We choose a system $\mathcal{P}_0$ of representatives of conjugacy classes of maximal parabolic subgroups.
Such a system is finite by \cite[Proposition~6.10]{Bowditch}.
The relative Cayley graph $\hat G$ is obtained by adding one edge of length 1 between every two elements in the same left coset of a parabolic subgroup in $\mathcal{P}_0$.} In other words, setting $\mathcal{P}_0=\{P_1,...,P_N\}$,
the relative Cayley graph $\hat G$ is the Cayley graph associated with the generating set $S\cup P_1\cup...\cup P_N$.
A (quasi-)geodesic in the relative Cayley graph $\hat G$ is called a relative (quasi-)geodesic.
This graph $\hat G$ is quasi-isometric to the coned-off graph introduced by Farb \cite{Farb} and is hyperbolic in the sense of Gromov.

A sequence $g_n$ in $G$ converges to a point $\xi$ in the Gromov boundary $\partial \hat G$ of $\hat G$ if and only if it converges to a conical limit point in the Bowditch compactification \cite[Section 8]{Bow98}.
We can thus identify $\partial \hat G$ with the set of conical limit points.
We refer to \cite{Tran} for more details on the comparison of these two boundaries.

\medskip
A very useful tool when studying the geometry of a relatively hyperbolic group is the notion of transition points on a geodesic.
\begin{definition}\label{deftransition}
Let $\gamma$ be a (finite or infinite) geodesic in the Cayley graph of $G$.
A point $v$ on $\gamma$ is said to be {\it $(\epsilon,R)$-deep} if there exists $g\in \Gamma$ and $P\in \mathcal{P}_0$ such that the $R$-neighborhood of $v$ in $\gamma$
is contained in the $\epsilon$-neighborhood of $gP$.
A point $v$ on $\gamma$ is called an \textit{$(\epsilon,R)$-transition point} if it is not $(\epsilon,R)$-deep.
\end{definition}

The following result of Hruska relates transition points and points on relative geodesics.
\begin{lem}\label{LemmaHruska}\cite[Proposition~8.13]{Hruska}
For every large enough $\epsilon$ and $R$, there exists $C$ such that the following holds.
Let $\alpha$ be a (finite or infinite) geodesic in $G$ and let $\hat \alpha$ be a relative geodesic with the same endpoints.
Then, any point on $\hat \alpha$ is within a distance at most $C$ of an $(\epsilon,R)$-transition point on $\alpha$.
Conversely, any $(\epsilon,R)$-transition point on $\alpha$ is within a distance at most $C$ of a point on $\hat \alpha$.
\end{lem}

\subsection{The Floyd distance and the Floyd boundary}
\label{SSFloyd} 
We first recall the definition of the Floyd distance and the Floyd boundary and their relation with the Bowditch boundary.
This boundary was introduced by Floyd in \cite{Floyd} and we also refer to \cite{KarlssonPoisson} and \cite{Karlssonfreesubgroups} for more details.

Let $G$ be a finitely generated group and let $\Gamma$ denote its Cayley graph associated with a finite generating set.
Let $f:\mathbb{N}\to \mathbb{R}$ be a function satisfying that the sum $\sum_{n\geq0}f_n$ is finite and that there exists $\lambda\in (0,1)$ such that $1\geq f_{n+1}/f_n\geq\lambda$ for all $n{\in}\mathbb{N}$.
The function $f$ is then called the {\it rescaling function} or the \textit{Floyd function}.
In the following, we will always choose an exponential Floyd function, that is the function $f$ will be of the form $f(n)=\lambda^n$ for some $\lambda\in (0,1)$.
Fix a basepoint $o\in \Gamma$ and rescale $\Gamma$ by declaring the length of an edge $\sigma$ to be $f(d(o,\sigma))$.
The induced shortpath metric on $\Gamma$ is called the {\it Floyd distance} with respect to the basepoint $o$ and Floyd function $f$ and is denoted by $\delta_{f,o}(.,.)$.
Whenever $f$ is of the form $f(n)=\lambda^n$, we will write $\delta_{\lambda,o}=\delta_{f,o}$ and if $o=1$, $\delta_\lambda=\delta_{f,o}$.

The Floyd compactification $\overline{\Gamma}^{\mathcal{F}}$ of $\Gamma$ is the Cauchy completion of $\Gamma$ endowed with the Floyd metric.
The Floyd boundary is then defined as $\partial_{\mathcal{F}}\Gamma=\overline{\Gamma}^{\mathcal{F}}\setminus \Gamma$. Different choices of basepoints yield bi-Lipschitz homeomorphisms of the Floyd compactifications. However, the topology  may  depend on the choice of the generating set and the rescaling function.
Keeping in mind $f(n)=\lambda^n$ and a choice of Cayley graph, we will also call $G\cup \partial_{\mathcal{F}}\Gamma$ the Floyd compactification of $G$ by abuse of language and we will write $\partial_{\mathcal{F}}G=\partial_{\mathcal{F}}\Gamma$.

The cardinality of the Floyd boundary is 0, 1, 2 or uncountable.
Moreover, it is 2 if and only if the group $G$ is virtually infinite cyclic, see \cite[Proposition~7]{Karlssonfreesubgroups}.
Following Karlsson, we say that the Floyd boundary is trivial if it is finite.
We will only have to deal with groups with non-trivial Floyd boundary.

Finally, as mentioned in the introduction, whenever the Floyd boundary is non-trivial, $G$ acts on it as a convergence action, see \cite[Thorem~2]{Karlssonfreesubgroups}.
Also, whenever the Floyd boundary is non-trivial, for any probability measure $\mu$ with finite first moment on $G$ and whose support generates $G$ as a semi-group, the random walk driven by $\mu$ almost surely converges to a point in the Floyd boundary.
Letting $\nu_{\mathcal{F}}$ be the law of the limit point, the pair $(\partial_{\mathcal{F}}G,\nu_{\mathcal{F}})$ is a model for the Poisson boundary, see \cite[Section~6, Corollary]{KarlssonPoisson}.

\medskip
We now assume that $G$ is non-elementary relatively hyperbolic.
We denote by $\partial_{\mathcal{B}}G$ its Bowditch boundary.
The following is due to Gerasimov.
\begin{theorem}\label{mapGerasimov}\cite[Map Theorem]{Gerasimov}
There exists $\lambda_0\in (0,1)$ such that for every $\lambda\in[\lambda_0,1)$, the identity of $G$ extends to a continuous and equivariant surjection $\phi$ from the Floyd compactification to the Bowditch compactification of $G$.
\end{theorem}

Actually, Gerasimov only stated the existence of the map $\phi$ for one Floyd function $f_0=\lambda_0^n$, but then
Gerasimov and Potyagailo proved that the same result holds for any Floyd function $f\geq f_0$, see \cite[Corollary~2.8]{GePoJEMS}.
They also proved that the preimage of a conical limit point is reduced to a single point and described the preimage of a parabolic limit point in terms of the action of $G$ on $\partial_{\mathcal{F}}G$, see precisely \cite[Theorem~A]{GePoJEMS}.
From now on, the parameter $\lambda$ will always be assumed to be contained in $[\lambda_0,1)$.

The Floyd distance can be transferred to a distance on the Bowditch boundary using the map $\phi$.
The resulting distance is called \textit{the shortcut metric} and we denote it by $\bar{\delta}_\lambda$.
It is the largest distance on the Bowditch boundary satisfying that for every $\xi ,\zeta\in \partial_\mathcal{F}G$,
\begin{equation}\label{shortcutsmallerFloyd}
    \bar{\delta}_\lambda(\phi(\xi),\phi(\zeta))\leq \delta_\lambda(\xi,\zeta).
\end{equation}
We refer to \cite[Section~4]{GePoGGD} for more details on its construction.
The next couple of lemmas will be used later on.

\begin{lem}[Visibility lemma] \label{karlssonlem}\cite{Karlssonfreesubgroups}
For every fixed $\lambda,c$, there exists a function $\varphi: \mathbb R_{\ge 0} \to \mathbb R_{\ge 0}$ such that for
any $v \in G$ and any $(\lambda,c)$-quasi-geodesic $\gamma$ in $\Gamma$, the following holds.
If
$\delta_{\lambda, v}(\gamma) \ge \kappa,$ then $d(v,
\gamma) \le \varphi(\kappa)$.
\end{lem}
Note that \cite{Karlssonfreesubgroups} only deals with geodesics, but the proof applies to quasi-geodesic with fixed parameters, see \cite{GePoJEMS} where this and more general cases are discussed.

The \textit{big shadow} $\Pi(g, R)$ at $g$ is the set of boundary points $\xi$ in the Bowditch boundary such that there exists a geodesic ray $[1,\xi]$ intersecting the ball $B(g,R)$.

\begin{lem}\label{lemma3.16PY}\cite[Lemma~3.14, Lemma~3.15]{PY}
For every large enough $\epsilon$ and $R$, the following holds.
Let $\xi$ be a conical limit point in the Bowditch boundary.
Consider a geodesic $\gamma$ between 1 and $\xi$ and consider any point $g$ on this geodesic.
There exists $C_1,C_2$ such that
$$\Pi(g, R)\subset B_{\bar{\delta}_\lambda}(\xi,C_1 r)$$
and if, in addition, $g$ is a $(\epsilon,R)$-transition point on $\gamma$, then
$$B_{\bar{\delta}_\lambda}(\xi,C_2 r)\subset  \Pi(g, R),$$
where $r=\lambda^{d(1,g)}$.
\end{lem}
Actually, \cite[Lemma~3.16]{PY} also states that the above statement is true for the Floyd distance, but we will only need to use it for the shortcut distance in estimating the lower bound of Hausdorff dimensions (see Lemma \ref{HDshortcutsmallerHDFloyd} and Proposition \ref{lower-boundharmonicFloyd}).

\begin{lem}\label{Lemma2.2Tanaka}
For every $g\in G$, there exists a constant $c_g>0$ such that the following inclusions hold.
For every point $\xi$ in the Floyd boundary and for every $r\geq 0$,
$$B_{\delta_\lambda}(g\xi,c_g^{-1}r)\subset gB_{\delta_\lambda}(\xi,r)\subset B_{\delta_\lambda}(g\xi,c_gr).$$
For every point $\xi$ in the Bowditch boundary,
$$B_{\bar{\delta}_\lambda}(g\xi,c_g^{-1}r)\subset gB_{\bar\delta_\lambda}(\xi,r)\subset B_{\bar\delta_\lambda}(g\xi,c_gr).$$
\end{lem}

\begin{proof}
First, a change a base point induces a bi-Lipschitz inequality for the Floyd distance : for every $x,y$ in the Floyd compactification, for any basepoints $o,o'\in G$,
$$\lambda^{d(o,o')}\leq \frac{\delta_{\lambda,o}(x,y)}{\delta_{\lambda,o'}(x,y)}\leq \lambda^{-d(o,o')},$$
see \cite[(2)]{PY}.
Now, by definition, the same is true of the shortcut distance, see precisely \cite[(3)]{PY}.

We only give the proof of the lemma for the Floyd distance, the proof for the shortcut distance is exactly the same.
Let $\zeta\in B_{\delta_\lambda}(\xi,r)$.
We need to prove that $g\zeta\in B_{\delta_\lambda}(g\xi,c_gr)$ for some $c_g$.
Note that $\delta_{\lambda,g}(g\zeta,g\xi)=\delta_{\lambda,1}(\zeta,\xi)$, so that by the above discussion,
$$\lambda^{d(1,g)}\leq \frac{\delta_{\lambda,1}(g\xi,g\zeta)}{\delta_{\lambda,1}(\xi,\zeta)}\leq \lambda^{-d(1,g)}.$$
This proves that $\delta_{\lambda,1}(g\xi,g\zeta)\leq \lambda^{-d(1,g)}r$ and so the right inclusion in the lemma holds for $c_g=\lambda^{-d(1,g)}$.
We immediately deduce the left inclusion, using $g^{-1}$.
\end{proof}

\subsection{Hausdorff dimension of measures}
Let $(X,d)$ be a metric space and $\kappa$ be a Borel measure on $X$.
\begin{definition}
The Hausdorff dimension of $\kappa$ is the smallest possible Hausdorff dimension of a set of full $\kappa$-measure :
$$Hdim(\kappa)=\inf \{Hdim(E),\kappa(E^c)=0\}.$$
When we want to insist on the choice of the distance, we will write $Hdim_d(\kappa)$.
\end{definition}

Evaluating the Hausdorff dimension of a set can be a difficult task, so the following characterization of the $Hdim(\kappa)$ as the essential supremum of the local dimensions of $\kappa$ is very useful.
Recall that the essential supremum $\kappa-\sup f$ of a function $f$ is defined as the infimum of the constants $C$ such that $f\leq C$ $\kappa$-almost everywhere.

\begin{prop}\cite[Corollary~8.2]{MSU}
Let $\kappa$ be a Borel measure on a metric space $X$.
Then,
$$Hdim(\kappa)=\kappa-\sup \liminf_{r\to 0}\frac{\log \kappa (B(x,r))}{\log r}.$$
\end{prop}

\begin{definition}
A measure $\kappa$ on a metric space $X$ is \textit{exact dimensional} if for $\kappa$-almost every $x$, the above $\liminf$ is a limit, that is, for $\kappa$-almost every $x$,
$$\lim_{r\to 0} \frac{\log \kappa (B(x,r))}{\log r}=Hdim(\kappa).$$
\end{definition}

We will also use the following notation :
$$\overline{Hdim}(\kappa)=\kappa-\sup \limsup_{r\to 0}\frac{\log \kappa (B(x,r))}{\log r}.$$
By definition, $Hdim(\kappa)\leq\overline{Hdim}(\kappa)$.
Let us say a few words about our strategy for evaluating the Hausdorff dimensions of harmonic measures $\nu$.
We will first prove that for $\nu$-almost every $x$, we have $\limsup_{r\to 0}\frac{\log \nu (B(x,r))}{\log r}\leq \frac{-1}{\log \lambda}\frac{h}{l}$, so that $\overline{Hdim}(\nu)\leq \frac{-1}{\log \lambda}\frac{h}{l}$.
We will then prove that for $\nu$-almost every $x$, we have $\frac{-1}{\log \lambda}\frac{h}{l}\leq \liminf_{r\to 0}\frac{\log \nu (B(x,r))}{\log r}$, so that
$\frac{-1}{\log \lambda}\frac{h}{l}\leq Hdim(\nu)$.
This will both prove that $Hdim(\nu)=\frac{-1}{\log \lambda}\frac{h}{l}$ and that $\nu$ is exact dimensional.

\section{Harmonic measures on the Floyd and the Bowditch boundaries}\label{FloyBowditchSection}
Let $G$ be a finitely generated non-elementary relatively hyperbolic group
and let $\mu$ be a probability measure with finite first moment on $G$.
Throughout this section, we consider the harmonic measure  denoted by $\nu_{\mathcal{B}}$ on the Bowditch boundary $\partial_{\mathcal B} G$ equipped with the shortcut distance $\bar\delta_\lambda$ and  the harmonic measure $\nu_{\mathcal{F}}$ on the Floyd boundary $\partial_{\mathcal F} G$ with the Floyd distance  $\delta_\lambda$.   

Our goal in this section is to prove the following theorem.

\begin{theorem}\label{theoremHausdorffharmonicFloydBowditch}
For any $\lambda\in [\lambda_0, 1)$  with $\lambda_0\in (0,1)$ given by Theorem \ref{mapGerasimov}, we have
$$Hdim_{\bar{\delta}_{\lambda}}(\nu_{\mathcal{B}})=\frac{-1}{\log \lambda} \frac{h}{l}$$
and, for any $\lambda\in (0,1)$,
$$Hdim_{{\delta}_{\lambda}}(\nu_{\mathcal{F}})=\frac{-1}{\log \lambda} \frac{h}{l}.$$
Moreover, those two measures are exact-dimensional.
\end{theorem}

\subsection{Sublinear deviation from transition points}
We denote by $\hat G$ the relative Cayley graph of $G$.
We fix large enough $\epsilon>0$ and $R>0$ satisfying the conclusions of Lemma~\ref{LemmaHruska} and Lemma~\ref{lemma3.16PY}.
Whenever $\alpha$ is a geodesic in the Cayley graph of $G$, we denote by $\Tr_{\epsilon,R}\alpha$, or simply by $\Tr\alpha$ the set of $(\epsilon,R)$-transition points on $\alpha$.
We denote by $\omega_n$ the random walk driven by $\mu$ at time $n$ and by $\omega_\infty$ its almost sure limit in the Bowditch boundary.

\begin{prop}\label{propsublineardeviationtransitionpoints}
With these notations, we have
$$\mathbb{P}\left (\sup d (\omega_n,\Tr \alpha)=o(n)\right )=1,$$
where the supremum is taken over all geodesics $\alpha$ from $1$ to $\omega_{\infty}$.
\end{prop}

\begin{proof}
We follow the strategy of \cite{Tiozzo}.
We introduce the function $f$ defined by
$$f(\omega)=\sup d (1,\Tr \hat{\alpha}),$$
where the supremum is taken over all geodesics $\hat{\alpha}$ between $\omega_{-\infty}$ and $\omega_{\infty}$,
where $\omega_{-\infty}$ is the limit of the reflected random walk, see \cite[Section~4.1]{MaherTiozzo}.

\begin{claim}\label{fmeasurable}
The function $f$ is measurable and is almost surely finite.
\end{claim}
\begin{proof}[Proof of the claim]
First, according to \cite[Theorem~1.1]{MaherTiozzo}, the exit points $\omega_{\infty}$ and $\omega_{-\infty}$ are conical limit points and their law $\nu$ and $\check{\nu}$ are non-atomic. Since bounded parabolic points are countable, it follows that $\omega_\infty$ and $\omega_{-\infty}$ almost-surely are distinct and so there exists a bi-infinite relative geodesic joining them.
Hence, the distance (in the Cayley graph) between $1$ and such a relative geodesic is finite.
According to Lemma~\ref{LemmaHruska}, any point on a relative geodesic is within a finite (and actually uniformly bounded) distance of a transition point on a geodesic in the Cayley graph, so $f$ is almost surely finite.

We now prove that $f$ is measurable.
Recall that $\partial \hat G$ is the Gromov boundary of the relative Cayley graph that we identify with the set of conical limit points in the Bowditch boundary $\partial_{\mathcal{B}}G$.
We just need to prove that the function
$$\tilde{f}:(\xi,\zeta)\in \partial \hat G\times \partial \hat G\mapsto \sup d(1,\Tr \alpha_{\xi,\zeta})$$ is measurable, where the supremum is taken over all word geodesics $\alpha_{\xi,\zeta}$ from $\xi$ to $\zeta$.
We follow the proof of \cite[Lemma~12]{Tiozzo}.
Note that $\tilde{f}$ takes the value $+\infty$ when $\xi=\zeta$.
Since there are no atoms at conical points, it can be extended to a function that we still denote by $\tilde{f}$ on the double Bowditch boundary $\partial_{\mathcal{B}}G\times \partial_{\mathcal{B}}G$.
To prove that $\tilde{f}$ is measurable, we just need to prove that for every $R\geq 0$, the set
$$\{(\xi,\zeta)\in \partial_{\mathcal{B}}G\times \partial_{\mathcal{B}}G,\tilde{f}(\xi,\zeta)> R\}$$
is measurable.

Recall the Bowditch compactification is a metrizable compact space containing the group $G$ as an open and dense set.
Choosing an arbitrary metric and taking a finite cover of $\partial_{\mathcal{B}}G$ made of balls of radius $1/k$, $k\in \mathbb{N}$, we construct a countable collection of open sets $U_n$ such that the sets $U_n\cap \partial_{\mathcal{B}}G$ form a countable base for the topology of $\partial_{\mathcal{B}}G$. Moreover, for each $R\geq 0$, only finitely many sets $U_n$ intersect the ball $B(1,R)$ and for each sequence $n_k$ going to infinity, the intersection $\cap_{k}U_{n_k}$ contains at most one point.
For fixed $R\geq 0$, say that a pair of open sets $(U,V)$ avoids the ball $B(1,R)$ if there exist $u\in U\cap G$ and $v\in V\cap G$ and there exists a geodesic $\gamma$ from $u$ to $v$ such that the ball $B(1,R)$ does not intersect $\Tr\gamma$.
Let us define $\mathcal{S}_R=\{(U_n,U_m) \text{ such that }(U_n,U_m)\text{ avoids the ball }B(1,R)\}$.
This is a countable collection of pairs of open sets.
By definition,
$$\{(\xi,\zeta)\in \partial_{\mathcal{B}}G\times \partial_{\mathcal{B}}G,\tilde{f}(\xi,\zeta)> R\}\subset \bigcap_{N\geq 1}\bigcup_{\overset{\min (n,m)\geq N}{(U_n,U_m)\in \mathcal{S}_R}}U_n\times U_m.$$
Conversely, consider $(\xi,\zeta)$ in this intersection and assume that $\tilde{f}(\xi,\zeta)<+\infty$.
Then, there are sequences of points $g_{n_k}$, respectively $h_{m_k}$, converging to $\xi$, respectively $\zeta$ and there is a geodesic $\gamma_k$ from $g_{n_k}$ to $h_{n_k}$ such that $B(1,R)$ does not intersect $\Tr\gamma_k$.
Now, since $\tilde{f}(\xi,\zeta)<+\infty$, there exists a geodesic from $\xi$ to $\zeta$ intersecting some big ball $B(1,R')$ for some $R'=R'_{\xi,\zeta}$.
Moreover, up to taking $R'$ large enough, all geodesics $\gamma_k$ also enter $B(1,R')$.
Thus, Arzel\'a-Ascoli Theorem allows us to choose a sub-sequence of geodesics $\gamma_{k_l}$ converging to a geodesic $\gamma$ from $\xi$ to $\zeta$, as $l$ tends to infinity.
We can also assume that the sub-geodesic of $\gamma_{k_l}$ contained in $B(1,R')$ is constant.
In particular, the limit geodesic $\gamma$ also satisfies that $B(1,R)$ does not intersect $\Tr\gamma$.
Hence, $\tilde{f}(\xi,\zeta)>R$.
This proves that
$$\{(\xi,\zeta)\in \partial_{\mathcal{B}}G\times \partial_{\mathcal{B}}G,\tilde{f}(\xi,\zeta)> R\}=\bigcap_{N\geq 1}\bigcup_{\overset{\min (n,m)\geq N}{(U_n,U_m)\in \mathcal{S}_R}}U_n\times U_m$$
and so $\{(\xi,\zeta)\in \partial_{\mathcal{B}}G\times \partial_{\mathcal{B}}G,\tilde{f}(\xi,\zeta)> R\}$
is measurable.
\end{proof}

Note that $f(T^n\omega)=d (\omega_n,\Tr [\omega_{-\infty},\omega_{\infty}])$, so that
$|f(T\omega)-f(\omega)|\leq d(1,\omega_1)$.
We deduce that the function $\omega\mapsto f(T\omega)-f(\omega)$ is integrable, since the random walk has finite first moment.
Thus, \cite[Lemma~7]{Tiozzo} shows that $\frac{1}{n}f(T^n\omega)$ almost surely converges to 0.
Finally,
consider a geodesic $\alpha_0$ from $1$ to $\omega_\infty$ and a geodesic $\hat\alpha_0$ from $\omega_{-\infty}$ to $\omega_\infty$.
Then, with probability one, there exists a transition point $x_n$ on $\hat{\alpha}_0$ such that $\frac{1}{n}d(\omega_n,x_n)$ converges to 0.
We now use that geodesic triangles are thin along transition points.
Precisely, according to \cite[Lemma~2.4]{GekhtmanDussaule}, $x_n$ is within a uniformly bounded distance of a transition point either on $\alpha_0$, or on a geodesic from 1 to $\omega_{-\infty}$.
Note that $\omega_n$ converges to $\omega_\infty$ and that $\frac{1}{n}d(1,\omega_n)$ almost surely converges to $l$.
Hence, $x_n$ also converges to $\omega_\infty$ and so for large enough $n$, it cannot be within a bounded distance of a geodesic from 1 to $\omega_{-\infty}$.
This proves that $\frac{1}{n}d(\omega_n,\Tr \alpha_0)$ also almost surely converges to 0.
Using again \cite[Lemma~2.4]{GekhtmanDussaule}, we see that $\frac{1}{n}\sup_{\alpha}d(\omega_n,\Tr \alpha)\leq \frac{1}{n}d(\omega_n,\Tr \alpha_0)+C$ for some uniform $C$.
This concludes the proof.
\end{proof}

We now deduce that the projection on parabolic subgroup is almost surely sublinear.
We choose a full subset $\mathcal{P}_0$ of representatives of conjugacy classes of maximal parabolic subgroups.
According to \cite[Proposition~6.10]{Bowditch}, such a set $\mathcal{P}_0$ is finite.
In the following, we will denote by $\mathcal{P}$ the set of all left cosets of elements of $\mathcal{P}_0$.
Let $U \in \mathcal{P}$.
Whenever $x,y\in G$, we set $d_U(x,y)=d(\pi_U(x),\pi_U(y))$, where $\pi_U$ is the projection on $U$ and where $d$ is the distance in the Cayley graph of $G$.
Our goal is to prove the following corollary.

\begin{cor}\label{coroparabolicdrift}
With the above notations, we have
$$\mathbb{P}\left (\sup_{U\in \mathcal{P}}d_U(1,\omega_n)=o(n)\right )=1.$$
\end{cor}

Before proving this corollary, note the following.

\begin{lem}
The sequence $f_n=\sup_{U\in \mathcal{P}}d_U(1,\omega_n)$ is sub-additive.
That is, for every $n,m\geq 1$, we have
$$f_{n+m}\leq f_n+f_m\circ T^n.$$
\end{lem}

\begin{proof}
Consider some $U\in \mathcal{P}$.
Then, $d_U(1,\omega_{n+m})\leq d_U(1,\omega_n)+d_U(\omega_n,\omega_{n+m})$, which we can rewrite $d_U(1,\omega_{n+m})\leq d_U(1,\omega_n)+d_{\omega_n^{-1}U}(1,\omega_n^{-1}\omega_{n+m})$.
In particular,
$$d_U(1,\omega_{n+m})\leq \sup_{U\in \mathcal{P}}d_U(1,\omega_n)+\sup_{U\in \mathcal{P}}d(1,\omega_n^{-1}\omega_{n+m})=f_n+f_m\circ T^n.$$
This is true for all $U\in \mathcal{P}$, which concludes the proof.
\end{proof}

According to Kingman's Theorem, $\frac{1}{n}f_n$ almost surely converges to some constant $l_{\mathcal{P}}$ that we call the parabolic linear drift.
We just need to prove that $l_{\mathcal{P}}=0$ to prove Corollary~\ref{coroparabolicdrift}.

\begin{proof}[Proof of Corollary \ref{coroparabolicdrift}]
By definition,
$$\sup_{U\in \mathcal{P}}d_U(\omega_n,\omega_{2n})=\sup_{U\in \mathcal{P}}d_{\omega_n^{-1}U}(1,\omega_n^{-1}\omega_{2n})=\sup_{U\in \mathcal{P}}d_U(1,\omega_n^{-1}\omega_{2n})$$
and since $\omega_{n}^{-1}\omega_{2n}$ and $\omega_n$ follow the same law, $\frac{1}{n}\sup_{U\in \mathcal{P}}d_U(\omega_n,\omega_{2n})$ also almost surely converges to $l_\mathcal{P}$.

Combining this with Proposition~\ref{propsublineardeviationtransitionpoints}, we see that for every positive $\epsilon$ and $\eta$, there exists $N_{\epsilon,\eta}$ such that with probability at least $1-\epsilon$, for all $n\geq N_{\epsilon,\eta}$, we have simultaneously
\begin{enumerate}[(a)]
    \item $\sup_\alpha d (\omega_n,\Tr \alpha)\leq \eta n$, where the supremum is taken over all geodesics $\alpha$ from $1$ to $\omega_\infty$,
    \item $\left |\sup_{U\in \mathcal{P}}d_U(\omega_n,\omega_{2n})-l_\mathcal{P}n\right |\leq \eta n$,
    \item $\left |\sup_{U\in \mathcal{P}}d_U(1,\omega_n)-l_\mathcal{P}n\right |\leq \eta n$.
\end{enumerate}
Fix $\epsilon$ and $\eta$ and set $N=N_{\epsilon,\eta}$.
For every $n\geq N$, there exists a transition point (with fixed parameters) $x_n$ on a geodesic $\alpha_n$ from $1$ to $\omega_\infty$, such that $d(\omega_n,x_n)\leq \eta n$.
Then, since for every $U\in \mathcal{P}$ and for every $x,y$, we have $d_U(x,y)\leq d(x,y)+c$ for some fixed $c$,
$$\sup_{U\in \mathcal{P}}d_U(1,x_n)\leq \sup_{U\in \mathcal{P}}d_U(1,\omega_n)+d(\omega_n,x_n)+c \leq l_\mathcal{P}n+2\eta n+c$$
and similarly,
$$\sup_{U\in \mathcal{P}}d_U(x_n,x_{2n})\leq l_\mathcal{P}n+3\eta n+2c.$$
Also,
$$\sup_{U\in \mathcal{P}}d_U(1,\omega_{2n})\leq \sup_{U\in \mathcal{P}}d_U(1,x_{2n})+\eta n+c.$$
By definition, $x_{2n}$ is a transition point on a geodesic $\alpha_{2n}$ from 1 to $\omega_\infty$.
Let $\delta>0$ and consider $U\in \mathcal{P}$ such that $d_U(1,x_{2n})\geq \sup_{U\in \mathcal{P}}d_U(1,x_{2n})-\delta$.
Then, according to \cite[Lemma~1.13, Lemma~1.15]{Sistoprojections}, the geodesic $\alpha_{2n}$ enters a fixed neighborhood of $U$ and the first point, respectively last point in this neighborhood is within a bounded distance of $\pi_U(1)$, respectively $\pi_U(x_{2n})$.
Now, $x_n$ also is a transition point on $\alpha_n$ with the same endpoints as $\alpha_{2n}$, so it is within a uniformly bounded distance of a transition point $\tilde{x}_n$ on $\alpha_{2n}$. There are two possibilities : the geodesic $\alpha_{2n}$ enters the neighborhood of $U$ either before or after $\tilde{x}_n$.
In the former case, we see that $d_U(1,x_{2n})\leq d_U(1,x_n)+c'$ for some fixed constant $c'$.
In the later case,
we see that $d_U(1,x_{2n})\leq d_U(x_n,x_{2n})+c'$.
In any case, we have that
$$d_U(1,x_{2n})\leq \max\{ d_U(1,x_n),d_U(x_n,x_{2n}) \}+c'\leq l_\mathcal{P}n+3\eta n +2c+c'.$$
Since this is true for every $\delta>0$, we have
$\sup_{U\in \mathcal{P}}d_U(1,x_{2n})\leq l_\mathcal{P}n+3\eta n +2c+c'$, hence finally, with probability at least $1-\epsilon$,
$$\sup_{U\in \mathcal{P}}d_U(1,\omega_{2n})\leq l_\mathcal{P}n+4\eta n +3c+c'.$$
Since $\epsilon$ and $\eta$ are arbitrary, this proves that almost surely,
$$\limsup_{n\to \infty}\frac{1}{2n}\sup_{U\in \mathcal{P}}d_U(1,\omega_{2n})\leq \frac{1}{2}l_{\mathcal{P}}.$$
Therefore, $l_{\mathcal{P}}\leq \frac{1}{2}l_\mathcal{P}$ and so $l_\mathcal{P}=0$.
\end{proof}

\subsection{Upper-bound for the Hausdorff dimension}
In this subsection, we first prove the following.

\begin{prop}\label{upper-boundharmonicFloyd}
Under the assumptions of Theorem \ref{theoremHausdorffharmonicFloydBowditch}, we have
$$\overline{Hdim}_{\bar{\delta}_{\lambda}}(\nu_{\mathcal{B}})\leq\frac{-1}{\log \lambda} \frac{h}{l}$$
and
$$\overline{Hdim}_{{\delta}_{\lambda}}(\nu_{\mathcal{F}})\leq\frac{-1}{\log \lambda} \frac{h}{l}.$$
\end{prop}

The proof is inspired from the proof of Le Prince for hyperbolic groups \cite{LePrinceupperbound}.
We first show that we only need to deal with the harmonic measure on the Floyd boundary.
Let $\phi$ be the map from the Floyd boundary to the Bowditch boundary given by Theorem~\ref{mapGerasimov}.
This map is surjective, equivariant and continuous.
Moreover, the preimage of a conical limit point consists of a singleton.

\begin{lem}\label{lemmapushforwardBowditch}
With the same notations, we have $\phi_*\nu_{\mathcal{F}}=\nu_{\mathcal{B}}$.
\end{lem}

\begin{proof}
Recall that a measure $\kappa$ on a set $X$ endowed with an action of $G$ is called $\mu$-stationary if it satisfies that for every measurable set $A\subset X$,
$$\kappa(A)=\sum_{g\in G}\kappa(g^{-1}A)\mu(g).$$
Combining \cite[Section~6, Theorem]{KarlssonPoisson} with \cite[Theorem~2.4]{Kaimanovichhyperbolic}, we get that the harmonic measure $\nu_{\mathcal{F}}$ on the Floyd boundary is $\mu$-stationary.
Since the map $\phi$ is equivariant, $\phi_*\nu_{\mathcal{F}}$ also is $\mu$-stationary.
Now, $\nu_{\mathcal{B}}$ is the only $\mu$-stationary measure on the Bowditch boundary, so that $\phi_*\nu_{\mathcal{F}}=\nu_B$.
This can be seen by verifying that the Bowditch compactification satisfies the conditions of \cite[Theorem~2.4]{Kaimanovichhyperbolic}, which is done for some examples of relatively hyperbolic groups in \cite[Section~9]{Kaimanovichhyperbolic}.
This can also be done directly by using results of Maher and Tiozzo \cite{MaherTiozzo}.
First, we prove by contradiction that any $\mu$-stationary measure on the Bowditch boundary has no atom. The following fact is well-known and we include a proof for completeness.
\begin{claim}\label{stationarywithoutatom}
Let $G$ be a group acting on a measured space $(X,\kappa)$ such that $\kappa$ is $\mu$-stationary and $G$ does not fix any finite set on $X$.
Then, $\kappa$ has not atom. 
\end{claim}
\begin{proof}[Proof of the claim]
If this were the case, then we would choose an atom $x\in X$ of maximal measure.
Since the support $\mathrm{supp}(\mu)$ generates the group $G$ as a semi-group, we would necessarily have $\kappa(g^{-1}x)=\kappa(x)$ for every $g$, for $\kappa$ is $\mu$-stationary, so
$$\kappa(x)=\sum_{g\in G}\kappa(g^{-1}x)\mu(g)\leq \sum_{g\in G}\kappa(x)\mu(g).$$
Hence, the orbit of $x$ would be finite, which is impossible since the group $G$ does not fix any finite set on $X$.
\end{proof}
We can apply this claim to our situation because $G$ is non-elementary and so it does not fix any finite set on the Bowditch boundary.
Thus, any $\mu$-stationary measure is uniquely defined by its restriction to the set of conical limit points and now \cite[Theorem~1.1]{MaherTiozzo} shows that there is a unique $\mu$-stationary probability measure on this set, which concludes the proof.
\end{proof}

We deduce the following lemma.
\begin{lem}\label{HDshortcutsmallerHDFloyd}
With the above notations, we have
$$\overline{Hdim}_{\bar{\delta}_{\lambda}}(\nu_{\mathcal{B}})\leq\overline{Hdim}_{{\delta}_{\lambda}}(\nu_{\mathcal{F}})$$
and
$$Hdim_{\bar{\delta}_{\lambda}}(\nu_{\mathcal{B}})\leq Hdim_{{\delta}_{\lambda}}(\nu_{\mathcal{F}}).$$
\end{lem}

\begin{proof}
By~(\ref{shortcutsmallerFloyd}), $\bar{\delta}_\lambda\leq \delta_\lambda$ on the set of conical limit points.
Hence, for every point $\xi$ such that $\phi(\xi)$ is a conical limit point,
$$\nu_{\mathcal{B}}\left (\phi(B(\xi,r))\right )\leq \nu_{\mathcal{B}}(B(\phi(\xi),r)).$$
Since $\nu_{\mathcal{B}}=\phi_*\nu_{\mathcal{F}}$,
$\nu_{\mathcal{B}}\left (\phi(B(\xi,r))\right )=\nu_{\mathcal{F}}\left (\phi^{-1}\circ \phi(B(\xi,r))\right )$.
We thus get
$$\nu_{\mathcal{F}}(B(\xi,r))\leq \nu_{\mathcal{F}}\left (\phi^{-1}\circ \phi(B(\xi,r))\right )\leq \nu_{\mathcal{B}}(B(\phi(\xi),r)).$$
This is true for every point $\xi$ such that $\phi(\xi)$ is a conical limit point.
Since the map $\phi$ is surjective and $\nu_{\mathcal{B}}$ gives full measure to the set of conical limit points, this concludes the proof.
\end{proof}

We thus only need to show that $$\overline{Hdim}_{{\delta}_{\lambda}}(\nu_{\mathcal{F}})\leq\frac{-1}{\log \lambda} \frac{h}{l}$$
to prove Proposition~\ref{upper-boundharmonicFloyd}.

We denote by $\omega_n$ the random walk driven by $\mu$ at time $n$. Let $\Omega=G^{\mathbb N_{\ge 0}}$ be the trajectory space of the random walk.   We also denote by $\mathbf{x}=(e,\omega_1,...,\omega_n,...)\in \Omega$ a sample path, i.e.\ a trajectory of the random walk starting from the identity $1$.
Recall that the random walk $\omega_n$ almost surely converges to a point $\omega_\infty$ in the Floyd boundary.

\begin{lem}\label{ZeroAsymJumpLem}
For $\mathbb P$-almost every $\mathbf{x}=(\omega_n)$, we have that
$$\frac{1}{n}d(\omega_n,\omega_{n+1})\underset{n\to +\infty}{\longrightarrow} 0.$$
\end{lem}

\begin{proof}
Note that $d(\omega_n,\omega_{n+1})=d(1,g_{n+1})$ where $g_n$ are the increments of the random walk.
In particular, the random variables $d(\omega_n,\omega_{n+1})$ are independent, identically distributed and are integrable, since $\mu$ has finite first moment.
If follows from the law of large numbers that
$$\frac{1}{n}\sum_{k=0}^nd(\omega_k,\omega_{k+1})\underset{n\to \infty}{\longrightarrow}L=\mathbb{E}(d(1,\omega_1))<+\infty.$$
In particular,
$$\frac{1}{n}d(\omega_n,\omega_{n+1})=\frac{1}{n}\sum_{k=0}^nd(\omega_k,\omega_{k+1})-\frac{1}{n}\sum_{k=0}^{n-1}d(\omega_k,\omega_{k+1})\underset{n\to \infty}{\longrightarrow} L-L=0.$$
This concludes the proof.
\end{proof}

As a corollary, we have the following.

\begin{lem}
For $\mathbb P$-almost every $\mathbf{x}=(\omega_n)$, we have that for every geodesic $\gamma_n$ from $\omega_n$ to $\omega_{n+1}$,
$$\frac{1}{n}d(1,\gamma_n)\underset{n\to \infty}{\longrightarrow}l.$$
\end{lem}

\begin{proof}
Let $\gamma_n$ be such a geodesic.
Then, we have $d(1,\gamma_n)\leq d(1,\omega_n)+d(\omega_n,\omega_{n+1})$ and similarly, we have $d(1,\omega_n)\leq d(1,\gamma_n)+d(\omega_n,\omega_{n+1})$.
Since $\frac{1}{n}d(1,\omega_n)$ almost surely converges to $l$ by~(\ref{equationdefl}) and $\frac{1}{n}d(\omega_{n},\omega_{n+1})$ almost surely converges to 0 by Lemma~\ref{ZeroAsymJumpLem}, we have that $\frac{1}{n}d(1,\gamma_n)$ also almost surely converges to $l$.
\end{proof}

For every $\epsilon, N>0$, we let $\Omega_\epsilon^N$ be the set of trajectories $\mathbf{x}$ such that for every $n\geq N$, we have
\begin{enumerate}[(a)]
    \item 
    $d(1,\gamma_n)\geq (l-\epsilon)n$
    \item
    $-\log \mu^{*n}(\omega_n)\leq (h+\epsilon)n$, and 
    \item
    $d(\omega_n,\omega_{n+1})\leq n$
\end{enumerate}
Observe that for every $\delta>0$, there exists $N_{\epsilon,\delta}$ such that
$$\mathbb{P}(\Omega_\epsilon^{N_{\epsilon,\delta}})\geq 1-\delta.$$ 
Indeed, this follows immediately from the almost sure convergence of the following limits  
$$\frac{1}{n}d(1,\gamma_n) \underset{n\to \infty}{\longrightarrow} l, \;\; \frac{1}{n}d(\omega_n,\omega_{n+1}) \underset{n\to \infty}{\longrightarrow} 0, \;\;   \frac{-1}{n}\log \mu^{*n}(\omega_n)\underset{n\to \infty}{\longrightarrow} h$$
where the last one is called the Shannon-McMillan-Breiman Theorem and is given by~(\ref{equationdefh}).

To simplify, we set $\Omega_{\epsilon,\delta}=\Omega_\epsilon^{N_{\epsilon,\delta}}$.
Also, for fixed $\mathbf{x}$, we let $C_{\mathbf{x}}^n$ be the set of trajectories $\mathbf{x}'$ such that $\omega_n'=\omega_n$.
Finally, for $\xi$ in the Floyd boundary and $r>0$, we set
$$D(\xi, r)=\{\mathbf{x}: \delta_{\lambda}(\omega_\infty,\xi)\leq r\}.$$

\begin{lem}\label{lemma2.1LePrince}
There exists a set $\Lambda_{\epsilon,\delta}\subset \Omega_{\epsilon,\delta}$ of measure at least $1-2\delta$ on which the quantity
$$\frac{\mathbb{P}(C_{\mathbf{x}}^n\cap \Omega_{\epsilon,\delta})}{\mu^{*n}(\omega_n)}$$
admits a positive limit.
In particular, for every $\mathbf{x}=(\omega_n)\in \Lambda_{\epsilon,\delta}$, we have
$$\limsup_{n\to \infty}\frac{1}{n}\log \mathbb{P}(C_{\mathbf{x}}^n\cap \Omega_{\epsilon,\delta})=\limsup_{n\to \infty}\frac{1}{n}\log \mu^{*n}(\omega_n).$$
\end{lem}

The proof can be found within the proof of \cite[Theorem~1.4.1]{KaimanovichHD}.
We rewrite it for convenience.

\begin{proof}
First note that $\mathbb{P}(C_{\mathbf{x}}^n)=\mu^{*n}(\omega_n)$.
So $\frac{\mathbb{P}(C_{\mathbf{x}}^n\cap \Omega_{\epsilon,\delta})}{\mu^{*n}(\omega_n)}$ is the conditional probability of $\Omega_{\epsilon,\delta}$ with respect to $C_{\mathbf{x}}^n$, that we denote by $\mathbb{P}(\Omega_{\epsilon,\delta}|C_{\mathbf{x}}^n)$.
We introduce the $\sigma$-algebra $\mathcal{A}^{\geq n}$ of sample paths determined by the coordinates $\omega_k$, $k\geq n$.
The tail $\sigma$-algebra $\mathcal{A}^{\infty}$ is then the intersection of the non-increasing sequence of $\sigma$-algebras $\mathcal{A}^{\geq n}$.
There is a projection from the path space to the tail boundary that we denote by $\mathbf{tail}$.
We refer to \cite{Kaimanovichboundaries} for more details.
We let $\mathbb{P}(A|\mathbf{tail}(\mathbf{x}))$ be the conditional probability $\mathbb{P}(A|\mathcal{A}^{\infty})$ evaluated at the sample path $\mathbf{x}$.

The Markov property and the conditional probabilities convergence theorem show that $\mathbb{P}(\Omega_{\epsilon,\delta}|C_{\mathbf{x}}^n)$ converges to $\mathbb{P}(\Omega_{\epsilon,\delta}|\mathbf{tail}(\mathbf{x}))$, see \cite[(1.4.4)]{KaimanovichHD}.
Now, let
$$\Lambda_{\epsilon,\delta}=\{\mathbf{x}\in \Omega_{\epsilon,\delta},\mathbb{P}(\Omega_{\epsilon,\delta}|\mathbf{tail}(\mathbf{x}))>0\}.$$
We have
$$\mathbb{E}(\mathbb{P}(\Omega_{\epsilon,\delta}|\mathbf{tail}(\mathbf{x})))=\mathbb{P}(\Omega_{\epsilon,\delta}).$$
Moreover,
$$\mathbb{E}(\mathbb{P}(\Omega_{\epsilon,\delta}|\mathbf{tail}(\mathbf{x})))=\mathbb{E}(\mathbb{P}(\Omega_{\epsilon,\delta}|\mathbf{tail}(\mathbf{x}))\cdot 1_{\Lambda_{\epsilon,\delta}})+\mathbb{E}(\mathbb{P}(\Omega_{\epsilon,\delta}|\mathbf{tail}(\mathbf{x}))\cdot 1_{\Omega_{\epsilon,\delta}^c}).$$
Since $\mathbb{P}(\Omega_{\epsilon,\delta}|\mathbf{tail}(\mathbf{x})))\leq 1$ for $\mathbb{P}$-almost every $\mathbf{x}$, we get
$$\mathbb{P}(\Omega_{\epsilon,\delta})=\mathbb{E}(\mathbb{P}(\Omega_{\epsilon,\delta}|\mathbf{tail}(\mathbf{x})))\leq \mathbb{P}(\Lambda_{\epsilon,\delta})+1-\mathbb{P}(\Omega_{\epsilon,\delta})$$
that we rewrite
$\mathbb{P}(\Lambda_{\epsilon,\delta})\geq 2 \mathbb{P}(\Omega_{\epsilon,\delta})-1$.
Therefore, $\mathbb{P}(\Lambda_{\epsilon,\delta})\geq 1-2\delta$, which concludes the proof.
\end{proof}

\begin{lem}\label{lemma2.2LePrince}
There exists $M\geq 0$ such that the following holds.
For fixed $\epsilon$ and $\delta$, for every $\mathbf{x}\in \Lambda_{\epsilon,\delta}$ and every $n\geq N_{\epsilon,\delta}$, we have
$$C_{\mathbf{x}}^n\cap \Omega_{\epsilon,\delta}\subset D(\omega_\infty,Mn\lambda^{n(l-\epsilon)}).$$
\end{lem}

\begin{proof}
Fix $\mathbf{x}\in \Lambda_{\epsilon,\delta}$ and let $\mathbf{x}'\in C_{\mathbf{x}}^n\cap \Omega_{\epsilon,\delta}$.
Then, $\omega_n=\omega'_n$.
Moreover, for every $m\geq N_{\epsilon,\delta}$, letting $\gamma'_m$ be a geodesic from $\omega'_m$ to $\omega'_{m+1}$, we have
$d(1,\gamma'_m)\geq (l-\epsilon)m$.
Thus there exists a path of length $d(\omega'_m,\omega'_{m+1})$ which stays at distance at least $(l-\epsilon)m$ from $1$.
Since we also have $d(\omega_m,\omega_{m+1})\leq m$, this proves that
\begin{equation}\label{estimateupperboundFloyd}
    \delta_{\lambda}(\omega'_m,\omega'_{m+1})\leq d(\omega'_m,\omega'_{m+1})\lambda^{(l-\epsilon)m}\leq m\lambda^{(l-\epsilon)m}.
\end{equation}
Consequently, for any $m>n$,
$$\delta_{\lambda}(\omega'_n,\omega'_m)\leq \sum_{k=n}^{m-1} k\lambda^{(l-\epsilon)k}\leq \sum_{k\geq 0}(n+k)\lambda^{(l-\epsilon)(n+k)}.$$
Note that
$$\sum_{k\geq 0}(n+k)\lambda^{(l-\epsilon)(n+k)}=n\lambda^{(l-\epsilon)n}\sum_{k\geq 0}\lambda^{(l-\epsilon)k}+\lambda^{(l-\epsilon)n}\sum_{k\geq 0}k\lambda^{(l-\epsilon)k}.$$
Since $\lambda<1$, both sums $\sum_{k\geq 0}\lambda^{(l-\epsilon)k}$ and $\sum_{k\geq 0}k\lambda^{(l-\epsilon)k}$ are finite.
Hence, there exists $M_0$ such that
$$\sum_{k\geq 0}(n+k)\lambda^{(l-\epsilon)(n+k)}\leq M_0n\lambda^{(l-\epsilon)n},$$
so that $\delta_{\lambda}(\omega_n,\omega_m)\leq M_0n\lambda^{(l-\epsilon)n}$.
This holds for every $m>n$, so we have
$$\delta_\lambda(\omega'_n,\omega'_{\infty})\leq M_0n\lambda^{(l-\epsilon)n}.$$
Clearly, we also have $\mathbf{x}'\in C_{\mathbf{x}}^n\cap \Omega_{\epsilon,\delta}$
and so $\delta_\lambda(\omega_n,\omega_{\infty})\leq M_0n\lambda^{(l-\epsilon)n}$.
Finally, since by definition of $C_{\mathbf{x}}^n$, $\omega_n=\omega_n'$, we get
$$\delta_\lambda(\omega_{\infty},\omega'_{\infty})\leq 2M_0n\lambda^{(l-\epsilon)n}.$$
This proves the lemma, setting $M=2M_0$.
\end{proof}

We can now finish the proof of Proposition~\ref{upper-boundharmonicFloyd}.

\begin{proof}
Our goal is to prove that for $\nu_{\mathcal{F}}$-almost every $\xi$,
\begin{equation}\label{nu_Fexactdimensional}
    \limsup_{r\to 0}\frac{\log \nu_{\mathcal{F}} (B(\xi,r))}{\log r}\leq \frac{-1}{\log \lambda}\frac{h}{l}.
\end{equation}
It is thus enough to prove that for $\mathbb{P}$-almost every trajectory $\mathbf{x}=(\omega_n)$, we have
$$\limsup_{r\to 0}\frac{\log \nu_{\mathcal{F}} (B(\omega_\infty,r))}{\log r}\leq \frac{-1}{\log \lambda}\frac{h}{l},$$
since $\nu_{\mathcal{F}}$ is the law of the random variable $\omega_\infty$. This shall follow from the   statement that for every $\epsilon$ and $\delta$, for every $\mathbf{x}\in \Lambda_{\epsilon,\delta}$, we have
$$\limsup_{r\to 0}\frac{\log \nu_{\mathcal{F}} (B(\omega_\infty,r))}{\log r}\leq \frac{-1}{\log \lambda}\frac{h+\epsilon}{l-\epsilon}.$$
Indeed,   $\mathbb{P}(\Lambda_{\epsilon,\delta})\geq 1-2\delta$ and $\epsilon$ and $\delta$ can be chosen arbitrarily small.

Let $\mathbf{x}\in \Lambda_{\epsilon,\delta}$.
We need to estimate $\limsup_{r\to 0}\frac{\log \nu_{\mathcal{F}} (B(\omega_\infty,r))}{\log r}$.
Since the function $x\in \mathbb{R}\mapsto Mx\lambda^{(l-\epsilon)x}$ is eventually decreasing,
we can replace $r$ with $Mn\lambda^{(l-\epsilon)n}$ and make $n$ go to infinity. Thus, it suffices to prove that
$$\limsup_{n\to \infty}\frac{\log \mathbb{P} (D(\omega_\infty,Mn\lambda^{(l-\epsilon)n}))}{\log \lambda (l-\epsilon)n}\leq \frac{-1}{\log \lambda}\frac{h+\epsilon}{l-\epsilon}.$$
In other words, we will prove that for every $\mathbf{x}\in \Lambda_{\epsilon,\delta}$,
$$\limsup_{n\to \infty}\frac{\log \mathbb{P} (D(\omega_\infty,Mn\lambda^{(l-\epsilon)n}))}{n}\geq -(h+\epsilon).$$
Using Lemma~\ref{lemma2.2LePrince}, we see that
$$\limsup_{n\to \infty}\frac{\log \mathbb{P} (D(\omega_\infty,Mn\lambda^{(l-\epsilon)n}))}{n}\geq \limsup_{n\to \infty}\frac{\log \mathbb{P} (C_{\mathbf{x}}^n\cap \Omega_{\epsilon,\delta})}{n}$$
and using Lemma~\ref{lemma2.1LePrince}, we see that
$$\limsup_{n\to \infty}\frac{\log \mathbb{P} (C_{\mathbf{x}}^n\cap \Omega_{\epsilon,\delta})}{n}=\limsup_{n\to \infty}\frac{\log \mu^{*n}(\omega_n)}{n}.$$
Finally, since $\mathbf{x}\in \Lambda_{\epsilon,\delta}\subset \Omega_{\epsilon,\delta}$, for large enough $n$, $\frac{\log \mu^{*n}(x_n)}{n}\geq -(h+\epsilon)$, so finally
$$\limsup_{n\to \infty}\frac{\log \mathbb{P} (D(\omega_\infty,Mn\lambda^{(l-\epsilon)n}))}{n}\geq -(h+\epsilon),$$
which concludes the proof.
\end{proof}

\subsection{Lower-bound for the Hausdorff dimension}
The goal of this subsection is the following.

\begin{prop}\label{lower-boundharmonicFloyd}
Under the assumptions of Theorem \ref{theoremHausdorffharmonicFloydBowditch}, we have
$$Hdim_{\bar{\delta}_{\lambda}}(\nu_{\mathcal{B}})\geq\frac{-1}{\log \lambda} \frac{h}{l}$$
and
$$Hdim_{{\delta}_{\lambda}}(\nu_{\mathcal{F}})\geq\frac{-1}{\log \lambda} \frac{h}{l}.$$
\end{prop}




\begin{proof}
We only prove the result for the measure $\nu_{\mathcal{B}}$ on the Bowditch boundary.
According to Lemma~\ref{HDshortcutsmallerHDFloyd}, the result for the harmonic measure on the Floyd boundary will then follows.

Proposition~\ref{propsublineardeviationtransitionpoints} shows there almost surely exists a transition point on a geodesic between 1 and $\omega_\infty$ such that $\frac{1}{n}d(\omega_n,x_n)$ converges to 0.
Also, $\frac{-1}{n}\log \mu^{*n}(\omega_n)$ almost surely converges to $h$ by~(\ref{equationdefh}) and $\frac{1}{n}d(1,\omega_n)$ almost surely converges to $l$ by~(\ref{equationdefl}).
For every $\epsilon>0$ and $N$, we let $\Omega_\epsilon^N$ be the set of trajectories $\mathbf{x}$ such that for every $n\geq N$,
\begin{enumerate}[(a)]
    \item 
    $d(\omega_n,x_n)\leq \epsilon n$
    \item
    $(l-\epsilon)n\leq d(1,x_n)\leq (l+\epsilon)n$
    \item
    $-\log \mu^{*n}(\omega_n)\geq (h-\epsilon)n$
\end{enumerate}
Then, for every $\epsilon$,
$\mathbb{P}(\cup_N\Omega_\epsilon^N)=1$.
Hence, there exists $N_\epsilon$ such that $\mathbb{P}(\Omega_{\epsilon}^{N_\epsilon})\geq 1-\epsilon$.
We set $\Omega_\epsilon=\Omega_{\epsilon}^{N_\epsilon}$.

We will both need to deal with the set $\Omega_\epsilon$ and the sets $\Omega_\epsilon^N$ in the following.
We fix $N$ and we fix a trajectory $\mathbf{x}\in \Omega_{\epsilon}^N$ and so we also fix the corresponding sequence of transition points $x_n$ on the geodesic between 1 and $\omega_\infty$.
Let $\Pi(g, R)$ be the big shadow at $g$, that is, the set of boundary points $\xi$ in the Bowditch boundary such that there exists a geodesic ray $[1,\xi]$ intersecting $B(g, R)$.

First of all, observe that for all $n$ large enough (i.e.\ bigger than $N$ and $N_\epsilon$),
$$\mathbb{P}(\mathbf{x}'\in\Omega_{\epsilon}\cap \{\omega_\infty' \in\Pi(x_n, R)\}) \leq \mathbb P(\mathbf{x}'\in \Omega_{\epsilon}: \omega'_{n}\in  B(x_n, 2R+3 n \epsilon)).$$

Indeed, if the limiting point $\xi'=\omega'_\infty$ lies in the shadow $\Pi(x_n, R)$, then there is a point $g_n$ on a geodesic from 1 to $\xi'$ entering the ball $B(x_n,R)$, so $d(1,g_n)$ is between $(l-\epsilon)n-R$ and $(l+\epsilon)n+R$.
For every $\mathbf x'\in \Omega_\epsilon^N$ with  transition points $x_n'$ on $[1, \omega_\infty']$, we also have that $d(1,x_n')$ is between $(l-\epsilon)n$ and $(l+\epsilon)n$, so we deduce that $d(x_n',g_n)\leq R+2\epsilon n$.
By (a), $d(\omega_n',x_n')\leq \epsilon n$, so we must have $d(\omega_n',x_n)\le d(\omega_n',g_n)+d(g_n,x_n)\leq 2R+3\epsilon n$ as desired.
 

By the defining property (c) of  $\Omega_{\epsilon}$, we have
$\mu^{*n}(\omega_n')\leq \exp(-n(h-\epsilon))$ for any $n>N_{\epsilon}$, so
\begin{align*}
    &\; \mathbb P(\mathbf{x}'\in \Omega_{\epsilon}, \omega'_{n}\in  B(x_n, 2R+3 n \epsilon)) \\
   \leq & \; \mathbb P(\mu^{*n}(\omega_n')\leq \exp(-n(h-\epsilon)), \omega'_{n}\in  B(x_n, 2R+3 n \epsilon))
\end{align*}
and since $\mu^{*n}$ is the law of $\omega_n'$, the right-hand side of this inequality can be written as the following sum
\begin{align*}
 \sum_{\underset{\mu^{*n}(u)\leq \exp(-n(h-\epsilon))}{u\in B(x_n,2R+3n\epsilon),}}\mu^{*n}(u).
\end{align*}
We thus get
$$\mathbb P(\mathbf{x}'\in \Omega_{\epsilon}, \omega'_{n}\in  B(x_n, 2R+3 n \epsilon)) \leq \sharp B(1, 2R+3n\epsilon) \exp(-n (h-\epsilon))$$
and since balls grow at most exponentially, there exists $v$ such that
$$\mathbb P(\mathbf{x}'\in \Omega_{\epsilon}, \omega'_{n}\in  B(x_n, 2R+3 n \epsilon)) \leq C \exp(3n\epsilon v) \exp(-n (h-\epsilon)).$$
Hence, for any $N$ and any fixed $\mathbf{x}$ in $\Omega_{\epsilon}^N$, we have
\begin{equation}\label{equationlower-boundshortcut}
\liminf_{n\to\infty} \frac{\log \mathbb{P}(\Omega_{\epsilon}\cap \Pi(x_n, R))}{-n} \ge h-\epsilon  -3v \epsilon
\end{equation}

Following~\cite{Tanaka}, we use conditional expectations with respect to the $\sigma$-algebra $\mathcal{S}$, which is the smallest $\sigma$-algebra such that the map $\mathbf{bnd}:\mathbf{x}\mapsto \omega_\infty$ is measurable.
Beware that the $\sigma$-algebra $\mathcal{S}$ and the $\sigma$-algebra $\mathcal{A}^{\infty}$ that we used in the proof of Proposition~\ref{upper-boundharmonicFloyd} can be different.
We refer to the discussion in the proof of \cite[Theorem~1.4.1]{KaimanovichHD} for more details.
We let $\mathbb{P}(A|\mathbf{bnd}(\mathbf{x}))$ be the conditional probability $\mathbb{P}(A|\mathcal{S})$ evaluated at the sample path $\mathbf{x}$.
This is denoted by $\mathbb{P}_{\mathbf{bnd}(\mathbf{x})}(A)$ in \cite{Tanaka}.
Since the harmonic measure $\nu_{\mathcal{B}}$ is the pushforward of the measure $\mathbb{P}$ by the map $\mathbf{bnd}$, 
one can define a family of conditional probabilities $\xi\in \partial_{\mathcal{B}}G\mapsto \mathbb{P}(A|\xi)$ such that for every measurable set $A\in \partial_{\mathcal{B}}G$ and every measurable set $B$ in the path-space of the random walk,
$$\mathbb{P}(B|\pi^{-1}(A))=\frac{1}{\nu(A)}\int_A\mathbb{P}(B|\xi)d\nu_{\mathcal{B}}(\xi).$$
We then define the set
$$F_{\epsilon}=\{\xi \in \partial_{\mathcal{B}}G, \mathbb{P}(\Omega_{\epsilon}|\xi)\geq \epsilon\}.$$
Desintegrating along $\nu_{\mathcal{B}}$, we have
\begin{align*}
    1-\epsilon \leq \mathbb{P}(\Omega_{\epsilon})&=\int_{\partial_{\mathcal{B}}G} \mathbb{P}(\Omega_{\epsilon}|\xi)d\nu_{\mathcal{B}}(\xi)\\
    &=\int_{F_\epsilon}\mathbb{P}(\Omega_{\epsilon}|\xi)d\nu_{\mathcal{B}}(\xi)+\int_{F_\epsilon^c}\mathbb{P}(\Omega_{\epsilon}|\xi)d\nu_{\mathcal{B}}(\xi)\\
    &\leq \nu_{\mathcal{B}}(F_{\epsilon})+\epsilon.
\end{align*}
Therefore, $\nu_{\mathcal{B}}(F_{\epsilon})\geq 1-2\epsilon$.
We now evaluate $\nu_{\mathcal{B}}(F_{\epsilon}\cap \Pi(x_n, R))$:
\begin{align*}
    \nu_{\mathcal{B}}(F_{\epsilon}\cap \Pi(x_n, R))&=\mathbb{P}(\omega_\infty\in F_{\epsilon}\cap \Pi(x_n,R))\\
    &\leq \mathbb{P}(\Omega_{\epsilon}\cap \Pi(x_n, R))+\mathbb{P}(\Omega_{\epsilon}^c\cap \{\omega_\infty \in F_{\epsilon}\cap \Pi(x_n,R)\})
\end{align*}
By definition, for $\xi\in F_{\epsilon}$, we have $\mathbb{P}(\Omega_{\epsilon}^c|\xi)\leq 1-\epsilon$, so
\begin{align*}
    \mathbb{P}(\Omega_{\epsilon}^c\cap \{\omega_\infty \in F_{\epsilon}\cap \Pi(x_n,R)\})&=\int_{F_{\epsilon}\cap \Pi(x_n,R)}\mathbb{P}(\Omega_{\epsilon}^c|\xi)d\nu_{\mathcal{B}}(\xi)\\
    &\leq (1-\epsilon)\nu_{\mathcal{B}}(F_{\epsilon}\cap \Pi(x_n, R)).
\end{align*}
Consequently,
$$\epsilon\cdot  \nu_{\mathcal{B}}(F_{\epsilon}\cap \Pi(x_n, R))\leq \mathbb{P}(\Omega_{\epsilon}\cap \Pi(x_n, R)).$$
Thus,~(\ref{equationlower-boundshortcut}) yields for every $N$ and every $\mathbf{x}\in \Omega_\epsilon^N$
$$\liminf_{n\to\infty} \frac{\log \nu_{\mathcal{B}}(F_{\epsilon}\cap \Pi(x_n, R))}{-n}\geq h-\epsilon  -3v \epsilon.$$
According to Lemma~\ref{lemma3.16PY}, there exists some constant $C>0$ such that
$$B(\omega_\infty,C r_n)\subset  \Pi(x_n, R)$$
where $r_n=\lambda^{d(1,x_n)}$ satisfies $\lambda^{(l+ \epsilon)n}\le r_n\le \lambda^{n(l- \epsilon)}$.
Therefore,
$$\liminf_{n\to\infty} \frac{\log\nu_{\mathcal{B}}(F_\epsilon\cap B(\omega_\infty, Cr_n))}{\log Cr_n}\ge  \frac{h-\epsilon-3v\epsilon}{ -\log \lambda  (l+\epsilon)}.$$
Thus, for any $N$ and any fixed $\mathbf{x}$ in $\Omega_{\epsilon}^N$, we have
$$\liminf_{r\to 0} \frac{\log\nu_{\mathcal{B}}(F_\epsilon\cap B(\omega_\infty, r))}{\log r}\ge  \frac{h-\epsilon-3v\epsilon}{ -\log \lambda  (l+\epsilon)}.$$
This is true for all $N$ and $\mathbb{P}(\cup_N \Omega_\epsilon^N)=1$. Hence, for $\nu_{\mathcal{B}}$-almost every $\xi$,
\begin{equation}\label{LocalDimEQ}
\liminf_{r\to 0} \frac{\log\nu_{\mathcal{B}}(F_\epsilon\cap B(\xi, r))}{\log r}\ge  \frac{h-\epsilon-3v\epsilon}{ -\log \lambda  (l+\epsilon)}.
\end{equation}

We can now conclude the proof, exactly like Theorem~1.1 is deduced from Theorem~3.3 in \cite{Tanaka}.
Let us give the details for completeness.

Consider the restriction $\nu_{{\mathcal{B}},\epsilon}$ of $\nu_{\mathcal{B}}$ to $F_{\epsilon}$.
Then,~(\ref{LocalDimEQ}) yields
$$Hdim (\nu_{{\mathcal{B}},\epsilon}) \geq \frac{h-\epsilon-3v\epsilon}{ -\log \lambda  (l+\epsilon)}$$ and so
$$Hdim(\nu_{\mathcal{B}})\geq \frac{h-\epsilon-3v\epsilon}{ -\log \lambda  (l+\epsilon)}.$$
In particular, the set
$$G_\epsilon=\left \{\xi \in \partial_{\mathcal{B}}G, \liminf_{r\to 0} \frac{\log\nu_{\mathcal{B}}(B(\xi, r))}{\log r}\ge  \frac{h-\epsilon-3v\epsilon}{ -\log \lambda  (l+\epsilon)}-\epsilon\right \}$$
has positive $\nu_{\mathcal{B}}$-measure.
We show that $G_\epsilon$ is $G$-invariant.
First, recall that the measure $\nu_{\mathcal{B}}$ is $\mu$-stationary.
Since $\mu$ generates $G$ as a semi-group, for any $g\in G$, there exists $n$ such that $\mu^{*n}(g^{-1})>0$ and $\mu^{*n}(g)>0$.
Note that $\nu_{\mathcal{B}}$ also is $\mu^{*n}$ stationary, so that for any measurable set $A$,
$$\sum_{g\in A}\nu_{\mathcal{B}}(gA)\mu^{*n}(g)=\nu_{\mathcal{B}}(A).$$
In particular, there exists $c_{g,\mu}>0$ such that 
$$c_{g,\mu}^{-1}\nu_{\mathcal{B}}\leq g^{-1}\nu_{\mathcal{B}}\leq c_{g,\mu}\nu_{\mathcal{B}}.$$
Together with Lemma~\ref{Lemma2.2Tanaka},  this implies that
$$\nu_{\mathcal{B}}(B(\xi,r))\leq c_{g,\mu}\nu_{\mathcal{B}}(B(g\xi,c_gr))
$$
for any $\xi\in G_\epsilon$.
By taking the limit inf, the constant $c_{g,\mu}$ disappears  and thus  $g\xi\in   G_\epsilon$. Hence,  $G_{\epsilon}$ is indeed $G$-invariant.

The measure $\nu_{\mathcal{B}}$ is ergodic and $\nu_{\mathcal{B}}(G_\epsilon)>0$.
Consequently, the set $G_{\epsilon}$ has full measure.
We thus get that for $\nu_{\mathcal{B}}$-almost every $\xi$,
$$\liminf_{r\to 0} \frac{\log\nu_{\mathcal{B}}( B(\xi, r))}{\log r}\ge  \frac{h-\epsilon-3v\epsilon}{ -\log \lambda  (l+\epsilon)}-\epsilon.$$
Since $\epsilon$ is arbitrary, this shows that for $\nu_{\mathcal{B}}$-almost every $\xi$,
\begin{equation}\label{nu_Bexactdimensional}
    \liminf_{r\to 0} \frac{\log\nu_{\mathcal{B}}( B(\xi, r))}{\log r}\ge \frac{-1}{\log \lambda} \frac{h}{l}
\end{equation}
and thus concludes the proof.
\end{proof}

We now conclude the proof of Theorem~\ref{theoremHausdorffharmonicFloydBowditch}.
\begin{proof}
By Propositions~\ref{upper-boundharmonicFloyd} and~\ref{lower-boundharmonicFloyd}, it remains to show the exact dimensionality of $\nu_{\mathcal{B}}$ and $\nu_{\mathcal{F}}$. According to~(\ref{nu_Fexactdimensional}) and Lemma~\ref{lemmapushforwardBowditch}, for $\nu_\mathcal{F}$-almost every $\xi$,
$$\limsup_{r\to 0} \frac{\log\nu_{\mathcal{B}}( B(\phi(\xi), r))}{\log r}\le \limsup_{r\to 0} \frac{\log\nu_{\mathcal{F}}( B(\xi, r))}{\log r} \le \frac{-1}{\log \lambda} \frac{h}{l}.$$ 
Similarly,~(\ref{nu_Bexactdimensional}) shows that for $\nu_{\mathcal{B}}$-almost every $\xi$,
$$\liminf_{r\to 0} \frac{\log\nu_{\mathcal{F}}( B(\xi, r))}{\log r}\ge \liminf_{r\to 0} \frac{\log\nu_{\mathcal{B}}( B(\phi(\xi), r))}{\log r}\ge  \frac{-1}{\log \lambda} \frac{h}{l}.$$
This shows that both $\nu_{\mathcal{B}}$ and $\nu_{\mathcal{F}}$ are exact dimensional.
\end{proof}

\section{Groups with infinitely many ends}\label{EndsSection}

In this section, we compute the Hausdorff dimension of the harmonic measure on the end boundary equipped with a visual metric

\subsection{The end compactification}

Let $\Gamma$ be an infinite, connected, locally finite  graph.
For every edge $e$ in $\Gamma$, we denote by $e^0$ its endpoints.
More generally, for every set of edges $E$, we denote by $E^0$ the set of vertices that are endpoints of an edge in $E$.
We will also write $\Gamma^0$ for the set of all vertices of $\Gamma$.

Let $E$ be a finite set of edges. Denote by $\Gamma\setminus E$ the spanning graph of the vertex set of $\Gamma^0\setminus E^0$: removing all edges  sharing one vertex with one edge in $E$. Let   $\mathcal C(E)$ be the set of   infinite components of $\Gamma\setminus E$. 
By definition, there exists at least one edge $e$ for every component $C\in \mathcal C(E)$ such that $e^0\cap C\ne \emptyset$ and $e^0\cap E^0\ne \emptyset$. 
Note that two components $C_1, C_2$ in $\mathcal C(E)$ are distinct if and only if every path between any two points $x\in C_1, y\in C_2$ intersects $E^0$.
Finally, two points $x, y$ are called \textit{separated} by $E$ if $x$ and $y$ lie in distinct components of $\Gamma\setminus E$. 

\medskip
We can define the \textit{end compactification} of $\Gamma$ as follows.
Consider the directed system $\mathcal F(\Gamma)$ of all finite set of edges in $\Gamma$ with $E< F$ if $E\subset F$.
There is a natural map from $\mathcal C(F)$ to $\mathcal C(E)$ induced by inclusions of infinite components.
The \textit{end boundary} $\partial_{\mathcal{E}} \Gamma$ is the inverse limit of the directed system $\mathcal C(E)$ over all finite set of edges $E$ in $\Gamma$.
By definition, a point $\xi\in \partial_{\mathcal{E}}\Gamma$ is a collection of infinite components $C_{E}(\xi)$ of $\Gamma\setminus E$ for every $E\in  \mathcal F(\Gamma)$, such that $C_{E}(\xi)\cap C_{E'}(\xi)$ is infinite for any two $E, E'$.
We call $\xi$ an \textit{end} of $\Gamma$.
For every $E$, the component $C_{E}(\xi)$ of $\mathcal C(E)$ is uniquely determined by $\xi$ and
by abuse of language, we say that $C_E(\xi)$ contains $\xi$, that we denote by $\xi\in C_E(\xi)$.
We can then extend the definition of separated pair of points to ends.
Two ends $\xi\ne \zeta$ are \textit{separated} by $E$ if $C_{E}(\xi)\ne C_{E}(\zeta)$, or equivalently if $\xi\notin C_E(\zeta)$.
Note that any two distinct ends are necessarily separated by some $E\in \mathcal{F}(\Gamma)$.

The end boundary $\partial_{\mathcal{E}}\Gamma$ defines a compactification of $\Gamma$ in the following way.
One can extend the discrete topology on the set of vertices $\Gamma^0$ to a metrizable topology on $\overline{\Gamma}^{\mathcal{E}}=\Gamma \cup \partial_{\mathcal{E}}\Gamma$ which makes it a compact space and such that $\Gamma$ is dense in $\overline{\Gamma}^{\mathcal{E}}$.
Moreover, a sequence of points $x_n\in \Gamma$ converges to an end $\xi\in \partial_{\mathcal{E}}\Gamma$ if and only if for every $E\in \mathcal F(\Gamma)$, we have that $x_n$ lies in $C_{E}(\xi)$ for all but finitely many $n$.

The topological closure $\bar C_{E}(\xi)$ of a component $C_E(\xi)$ in the compactification $\Gamma\cup\partial_{\mathcal{E}}\Gamma$ is the union of $C_{E}(\xi)$ with all $\zeta \in \partial_{\mathcal E} \Gamma$ for which $C_{E}(\xi)= C_{E}(\zeta)$.
Hence, a component $C$ contains $\xi$ if $\xi\in \bar C$.
Let $C_{E_i}(\xi)$ be a sequence of strictly shrinking components, that is, \ $C_{E_{i+1}}(\xi)\subset C_{E_i}(\xi)$.
Then their closures $\bar C_{E_i}(\xi)$ yield a neighborhood basis of $\xi$. 

Let $\xi$ be an end.
By definition, there exists a sequence of finite subsets $E_n$ such that $C_{E_{n+1}}(\xi)\subset C_{E_n}(\xi)$.
Following Woess \cite{Woessends}, we say that $\xi$ is a \textit{thin end} if the sets $E_n$ can be chosen so that $\sup_n \diam(E_n)$ is finite.
We say that it is \textit{$M$-thin} if $\sup_n \diam(E_n)\leq M$.

\medskip
Then end compactification of a finitely generated group $G$ is the end compactification of its Cayley graph with respect to a finite generating set. The quasi-isometry extends to a homeomorphism between end boundary, so the topology of the end boundary does not depend on the choice of this generating set  and so it is well defined. We denote by $\partial_{\mathcal{E}}G$ the end boundary of $G$. This compactification was first introduced by Freudenthal and is also called the Freudenthal compactification. We refer to \cite{Freudenthal} for more details.

An infinitely-ended group is called \textit{accessible} if it admits a splitting over finite edge groups  as a finite graph of one ended or finite vertex groups. Finitely presented group  are accessible \cite{Dun85}.
The accessibility is a quasi-isometric invariant by the following graphical characterization \cite{TW93}.
An infinitely-ended group is accessible if and only if there exists $k>0$ such that any two distinct ends are separated by $k$ edges in a Cayley graph.  

Taking limits of geodesics and using Arzel\'a-Ascoli Theorem, we see that
the end boundary is a visual boundary: any end is connected to any point $x$ in the group by an infinite geodesic and any two distinct ends are connected by a bi-infinite geodesic.
Finally, by results of Stallings \cite{Stallings}, $G$ acts on $\partial_{\mathcal{E}}G$ as a convergence group, see also \cite[Lemma~5.1]{Bowditchconvergenceonends}.

\subsection{Visual metrics}
Let $G$ be a finitely generated group with infinitely many ends.
In \cite{CandelleroGilchMuller}, Candellero, Gilch and M\"uller defined a \textit{visual metric} on the set of ends of a free product.
We extend their definition to our situation.
Fix a basepoint $o\in G$ and fix $0<\lambda<1$. 
We define a distance $\rho_{o,\lambda}$ on $\partial_{\mathcal{E}}G$ as follows.
Consider the sequence of finite edge sets $B_n$, where $B_n$ is the edge set of the subgraph spanned by the vertices in the closed ball of radius $n$ around $o$.
The inverse limit of the directed system $\{\mathcal C(B_n)\}$ is homeomorphic to $\partial_{\mathcal{E}}\Gamma$.

Let $\xi, \zeta$ be two distinct ends. Let $n$ be the minimal integer such that $\xi$ and $\zeta$ belongs to different components in $\mathcal C(B_n)$, then define $\rho_{\lambda,o}(\xi, \zeta)=\lambda^{n}$.
When the basepoint is the neutral element $1$ of the group, we will simply write $\rho_{\lambda}(\xi,\zeta)$.

By definition, the visual metric is \textit{ultrametric}: for any triple of points $x,y,z\in\partial_{\mathcal E} \Gamma$, $$\rho_\lambda(x, y)\le \max\{\rho_\lambda(x,z), \rho_\lambda(z, y)\}$$

The next lemma follows directly from the definition.

\begin{lem}\label{lemchangeofbasepointvisualmetric}
For every $x,y \in \partial_{\mathcal E}\Gamma$ in the end boundary and for any choice of basepoints $o,o'\in G$,
$$\lambda^{d(o,o')}\leq \frac{\rho_{\lambda,o}(x,y)}{\rho_{\lambda,o'}(x,y)}\leq \lambda^{-d(o,o')}.$$
\end{lem}

As a consequence, we get the following, which is proved exactly like Lemma~\ref{Lemma2.2Tanaka}.
\begin{lem}\label{Lemma2.2Tanakaends}
For every $g\in G$, there exists a constant $c_g$ such that
for every end $\xi$ and for every $r\geq 0$,
$$B_{\rho_\lambda}(g\xi,c_g^{-1}r)\subset gB_{\rho_\lambda}(\xi,r)\subset B_{\rho_\lambda}(g\xi,c_gr).$$
\end{lem}

It is well known that the Floyd boundary covers the end boundary, see for example \cite[Proposition~11.1]{GGPY} and \cite{Karlssonfreesubgroups}.
The following lemma allows us to compare the Floyd distance with the visual distance.

\begin{lem}\label{FloydDomVisu}
The identity of $G$ extends to a 1-Lipschitz surjective equivariant map $\phi$ from the Floyd boundary to the end boundary with the same parameter $\lambda\in (0,1)$:
$$\rho_\lambda(\phi(\xi), \phi(\zeta))\le \delta_\lambda(\xi, \eta)$$ for any $\xi, \zeta\in \partial_{\mathcal F} G$. 
\end{lem}

\begin{proof}
Consider $x,y\in G$ such that $\rho_\lambda(x,y)=\lambda^n$, where $n$ is the minimal integer such that $x, y$ are contained  in distinct components of $\mathcal C(B_n)$.
Then any path from $x$ to $y$ has to intersect $B(1, n)$, so  $\delta_\lambda(x,y)\ge \lambda^n$.
Hence, for any $x,y\in G$,
\begin{equation}\label{Floydbigervisualgroup}
    \delta_\lambda(x,y)\ge\rho_\lambda(x,y).
\end{equation}
Consider now $\xi \in \partial_\mathcal{F}G$ and let $x_n$ be a sequence in $G$ converging to $\xi$.
Then, $x_n$ is Cauchy for the Floyd distance and so~(\ref{Floydbigervisualgroup}) shows it is also Cauchy for the visual distance.
This proves that $x_n$ converges to a point $\phi(\xi)\in \partial_{\mathcal{E}}G$, which is uniquely determined by $\xi$.
By construction,~(\ref{Floydbigervisualgroup}) extends to points in the boundary and $\phi$ is equivariant.
\end{proof}

\subsection{The end boundary of accessible groups}\label{SSRelHyp}
As explained in Introduction, any infinitely-ended group is relatively hyperbolic.  If it is accessible, the action on its end boundary is geometrically finite. Precisely, the peripheral structure is given by the ``terminal'' splitting of the accessible group   as a finite graph of groups over finite edge groups so that  every vertex groups are either one-ended or finite.    So for accessible groups, the end boundary is homeomorphic to the Bowditch boundary, whose construction is briefly recalled below.

Let $T$ be the corresponding Bass-Serre tree of the above terminal splitting. Following Bowditch \cite{Bow98}, we can put a compact metrizable topology on $T^0\cup \partial_{\mathcal E} T$, for $T$ is a fine hyperbolic graph. A similar construction is also given for any (non-)locally finite graph by Cartwright-Soardi-Woess \cite{CSW}. 

As a perfect compact space, the Bowditch boundary or the end boundary is homeomorphic to the subspace of $T^0\cup \partial_{\mathcal E} T$ minus the isolated points coming from the vertices with finite stabilizer. According to the definition of a geometrical finite action,  every point in the Bowditch boundary is either conical or bounded parabolic. 
The set of conical points are exactly $\partial_{\mathcal E} T$ and bounded parabolic points are the vertices in $T$ with infinite stabilizer, which are the subset of ends in $\partial_{\mathcal E} G$ corresponding to the left cosets of stabilizer in the Caylay graph of $G$.  Finally, let us remark that these two types of limit points are precisely thin end and thick end introduced and studied in \cite{TW93}. The proof of   this fact can be found in \cite[Proposition 7.8]{GGPY}

\subsection{Hausdorff dimension of the harmonic measure}
In this section, we consider a finitely generated group $G$ with infinitely many ends
and we consider a probability measure $\mu$ with finite first moment on $G$.
We denote by $\omega_n$ the random walk driven by $\mu$.
It is a classical fact that $\omega_n$ almost surely converges to an end $\omega_{\infty}$, and that denoting by $\nu_{\mathcal{E}}$ the law of $\omega_{\infty}$, the end boundary $(\partial_{\mathcal{E}}G,\nu_{\mathcal{E}})$ is a model for the Poisson boundary.
See for example \cite[Theorem~8.4]{Kaimanovichhyperbolic}. We call $\nu_{\mathcal{E}}$ the harmonic measure on $\partial_{\mathcal{E}}G$.

\begin{prop}\label{propPoissonends}
There exists $M>0$ such that $\nu_{\mathcal{E}}$ gives full measure to the set of $M$-thin ends.
\end{prop}

\begin{proof}
It is proved in \cite[Theorem~4.1]{Woessends} that the set of ends of a locally finite graph with infinitely many ends can be decomposed into the union $\Omega_0\cup \Omega'$, where $\Omega_0$ is a dense set in the set of ends consisting of $M$-thin ends.
When the graph is the Cayley graph of a finitely generated group, the set $\Omega_0$ can be constructed as follows.
As explained in the introduction, the group $G$ splits as an amalgamated product $A*_CB$ or an HNN extension $A*_C$.
In the former case, $G$ is hyperbolic relative to the conjugates of $A$ and $B$ and in the later case, it is hyperbolic relative to the conjugates of $A$.
In both cases, every element of $G$ can be written with elements of $A$, $B$ and $C$ with a normal form, see \cite[(9.2),(9.4)]{Woessends} and $\Omega_0$ can be described as the set of infinite words with respect to this normal form, see \cite[(9.3),(9.5)]{Woessends}. 
Moreover, the set $\Omega'$ is constructed as the union of translates of the set of ends of $A$ and $B$, see the remarks after \cite[(9.3),(9.5)]{Woessends}.
We can thus construct a continuous, surjective and equivariant map $\psi$ from the set of ends $\partial_{\mathcal{E}}G$ to the Bowditch boundary of $G$ with respect to the relatively hyperbolic structure described above.
The map $\psi$ is obtained by collapsing the translate $g \partial_{\mathcal{E}}A$ of the set of ends of $A$ to the point $g\alpha$ in the Bowditch boundary, where $\alpha$ is the parabolic limit point fixed by $A$ and similarly with $B$.
It follows from this construction that $\Omega_0$ is mapped to conical limit points.
Note that the measure $\nu_{\mathcal{E}}$ is $\mu$-stationary and since $\psi$ is equivariant, the pushforward $\psi_*\nu_{\mathcal{E}}$ also is $\mu$-stationary on the Bowditch boundary.
As explained in the proof of Lemma~\ref{lemmapushforwardBowditch}, the harmonic measure $\nu_{\mathcal{B}}$ on the Bowditch boundary is the unique $\mu$-stationary measure, hence $\psi_*\nu_{\mathcal{E}}=\nu_{\mathcal{B}}$.
Moreover, $\nu_{\mathcal{B}}$ gives full measure to conical limit points.
Consequently,
$$\nu_{\mathcal{E}}(\Omega')\leq \nu_{\mathcal{E}}(\psi^{-1}(\psi(\Omega')))=\nu_{\mathcal{B}}(\psi(\Omega'))=0.$$
This concludes the proof.
\end{proof}

The following definition is inspired by the work of Derriennic \cite{Derriennic} in   free groups.

\begin{definition}\label{defbottleneck1}
Let $x,y\in G\cup \partial_\mathcal{E}G$ and let $M\geq 0$.
We say that a set $U\subset G$ is a \textit{$M$-bottleneck} between $x$ and $y$ if $\diam (U)\leq M$ and any path from $x$ to $y$ has to pass through $U$.
\end{definition}

Similar sets are called transitional sets by Derriennic in \cite{Derriennic}.
However, to avoid confusion with the terminology "transition points" in relatively hyperbolic groups, we used the name bottleneck, which will also be more suited to our use later.
The next lemma follows from our definitions.

\begin{lem}\label{existencebottlenecks}
Let $\xi$ be a $M$-thin end for $M>0$.
Then there exists an infinite sequence of distinct sets $E_n$ in $G$ with ~$\sup_{n
\ge 1}\{\diam (E_n)\}\leq M$~ such that any path from 1 to $\xi$ has to pass successively through each $E_n$.

Moreover, for any $x\neq \xi\in G\cup \partial_{\mathcal{E}}G$, there exists $n_0>0$ such that any path from $x$ to $\xi$ has to pass successively through each $E_n$, for $n\geq n_0$.
In particular,  for all but finitely many  $n>0$, the sets $E_n$ are $M$-bottlenecks between $x$ and $\xi$.
\end{lem}

Our goal in the remainder of this section is to compute the Hausdorff dimension of the harmonic measure with respect to the visual distance.
Precisely, we prove the following theorem.

\begin{theorem}\label{theoremHausdorffharmonicends}
Let $(\partial_{\mathcal{E}}G,\nu_{\mathcal{E}})$ be as above and let $h, l$ be the entropy and the rate of escape of the $\mu$-random  walk on $G$. Then,  for any $\lambda\in (0,1)$,
$$Hdim_{\rho_{\lambda}}(\nu_{\mathcal{E}})=\frac{-1}{\log \lambda} \frac{h}{l}.$$
Moreover, $\nu_{\mathcal{E}}$ is exact dimensional.
\end{theorem}

We will follow the strategy that we used for the harmonic measures on the Bowditch and the Floyd boundaries.
We first give an upper bound.

\begin{prop}\label{upper-boundharmonicends}
For $\nu_{\mathcal{E}}$-almost every $\xi$ in $\partial_{\mathcal{E}}G$,
$$\limsup_{r\to 0}\frac{\log \nu_{\mathcal{E}}(B_{\rho_\lambda}(\xi,r))}{\log r}\leq \frac{-1}{\log \lambda} \frac{h}{l}.$$
\end{prop}

\begin{proof}
The proof of Proposition~\ref{upper-boundharmonicFloyd} for the Floyd distance again applies here.
The only place where the Floyd distance is used there is in the estimate~(\ref{estimateupperboundFloyd}), which states that whenever $x,x'$ are joined by a geodesic which stays at distance at least $m_1$ from $1$ and which satisfies that $d(x,x')\leq m_2$, we have
$$\delta_{\lambda}(x,x')\leq m_2\lambda^{m_1}.$$
This is again true replacing the Floyd distance $\delta_{\lambda}$ by the visual distance $\rho_{\lambda}$, since we have the better estimate
$$\rho_{\lambda}(x,x')\leq \lambda^{m_1}.$$

Also note that for $\lambda\geq \lambda_0$, we can give a direct proof.
We use the map $\phi$ given by Lemma~\ref{FloydDomVisu}.
By \cite[Theorem~8.3]{Kaimanovichhyperbolic}, $\nu_{\mathcal{E}}$ is the only $\mu$-stationary measure on $\partial_{\mathcal{E}}G$.
Hence, $\phi_*\nu_{\mathcal{F}}=\nu_{\mathcal{E}}$.
The result thus follows from the same result for the measure $\nu_{\mathcal{F}}$, which is given by~(\ref{nu_Fexactdimensional}).
\end{proof}

To obtain a lower bound, we will use the following result, which says that the random walk almost surely sublinearly tracks bottlenecks.

\begin{prop}\label{propsublineardeviationbottlenecks}
There exists $M\geq 0$ such that for $\mathbb P$-almost every $\mathbf x=(\omega_n)\in \Omega$, there exists a sequence of $M$-bottlenecks $U_n$ between 1 and $\omega_\infty$ satisfying that
$$\frac{1}{n}d(\omega_n,U_n)\underset{n\to \infty}{\longrightarrow}0.$$
\end{prop}

\begin{proof}
We fix $M\geq 0$ as in Proposition~\ref{propPoissonends}.
For every $x,y\in G\cup \partial_{\mathcal{E}}G$, we denote by $\Bn(x,y)$ the set of $M$-bottlenecks between $x$ and $y$.
We introduce the function $f$ defined by
$f(\omega)=d(1,\Bn(\omega_{-\infty},\omega_\infty))$,
where $\omega_{-\infty}$ is the limit of the reflected random walk in $\partial_{\mathcal{E}}G$.
Since the measure $\nu_{\mathcal{E}}$ is stationary, it is non-atomic.
Indeed, we can apply the Claim~\ref{stationarywithoutatom}, for $G$ cannot fix any finite set on $\partial_{\mathcal{E}}G$ since it is non-amenable, see \cite[Theorem~2.3]{Woessends}.
Hence, $\omega_{-\infty}$ and $\omega_\infty$ are almost surely distinct.
According to Proposition~\ref{propPoissonends}, $\omega_{-\infty}$ and $\omega_\infty$ are almost surely $M$-thin ends, so Lemma~\ref{existencebottlenecks} shows that $\Bn(\omega_{-\infty},\omega_\infty)$ is non-empty.
The proof of the Claim~\ref{fmeasurable} thus shows that $f$ is measurable and is almost surely finite.

Note that $f(T^n\omega)=d(\omega_n,\Bn(\omega_{-\infty},\omega_\infty))$ and so
$|f(T\omega)-f(\omega)|\leq d(1,\omega_1)$.
Hence, \cite[Lemma~7]{Tiozzo} shows that $\frac{1}{n}f(T^n\omega)$ almost surely converges to 0.
Thus, there almost surely exists a sequence of $M$-bottlenecks $U_n$ between $\omega_{-\infty}$ and $\omega_\infty$ such that
$\frac{1}{n}d(\omega_n,U_n)$ converges to 0.

To conclude, we just need to show that for large enough $n$, $U_n$ is also a $M$-bottleneck between 1 and $\omega_{\infty}$.
Since $d(1,\omega_n)$ almost surely converges to $l$, we can assume that for large enough $n$, $d(1,\omega_n)\geq (l-\epsilon)n$ and so $d(1,U_n)$ goes to infinity.
Fix a path $\alpha$ from $\omega_{-\infty}$ to 1.
Then for large enough $n$, say $n\geq n_0$, $\alpha$ does not intersect $U_n$.
Consider now any path $\beta$ from 1 to $\omega_{\infty}$.
Concatenating $\alpha$ and $\beta$ yields a path from $\omega_{-\infty}$ to $\omega_\infty$ which thus crosses $U_n$.
By construction, we necessarily have that $\beta$ intersects $U_n$ for $n\geq n_0$. 
\end{proof}

We can now prove the following result.

\begin{prop}\label{lower-boundharmonicends}
For $\nu_{\mathcal{E}}$-almost every $\xi$ in $\partial_{\mathcal{E}}G$,
$$\liminf_{r\to 0}\frac{\log \nu_{\mathcal{E}}(B_{\rho_\lambda}(\xi,r))}{\log r}\geq \frac{-1}{\log \lambda} \frac{h}{l}.$$
\end{prop}

\begin{proof}
We choose a sequence of points $x_n\in U_n$, where $U_n$ is a sequence of $M$-bottlenecks between 1 and $\omega_{\infty}$ given by Proposition~\ref{propsublineardeviationbottlenecks}.
For every $\epsilon>0$ and $N$, we let $\Omega_\epsilon^N$ be the set of trajectories $\mathbf{x}=(\omega_n) \in \Omega$ such that for every $n\geq N$,
\begin{enumerate}[(a)]
    \item 
    $d(\omega_n,x_n)\leq \epsilon n$,
    \item
    $(l-\epsilon)n\leq d(1,x_n)\leq (l+\epsilon)n$, and
    \item
    $-\log \mu^{*n}(\omega_n)\geq (h-\epsilon)n$.
\end{enumerate}
Almost surely, $\frac{1}{n}d(\omega_n,x_n)$ converges to 0,  $\frac{1}{n}d(1,\omega_n)$ converges to $l$ by~(\ref{equationdefl}), and $\frac{-1}{n}\log \mu^{*n}(\omega_n)$ converges to $h$ by~(\ref{equationdefh}).
Hence, there exists $N_\epsilon$ such that $\mathbb{P}(\Omega_{\epsilon}^{N_\epsilon})\geq 1-\epsilon$.
We set $\Omega_\epsilon=\Omega_{\epsilon}^{N_\epsilon}$.

We fix $N$ and we fix a trajectory $\mathbf{x}\in \Omega_{\epsilon}^N$ and so we also fix the corresponding sequence of points $x_n$.
For $g\in G$, we define the \textit{partial shadow} $\mho(g,M)$ to be the set of $\xi\in \partial_{\mathcal{E}}G$ such that $g$ lies in a $M$-bottleneck between 1 and $\xi$.
For $n\geq N$ and $n\geq N_{\epsilon}$, we have
$$\mathbb{P}(\mathbf{x}'\in\Omega_{\epsilon}\cap \{\omega_\infty' \in\mho(x_n, M)\}) \leq \mathbb P(\mathbf{x}'\in \Omega_{\epsilon}, \omega'_{n}\in  B(x_n, 4M+3 n \epsilon)).$$
Indeed, assume that $x_n$ lies in a $M$-bottleneck $V$ between 1 and $\omega_{\infty}'$.
Fix a geodesic from 1 to $\omega_{\infty}'$.
This geodesic enters $V$ at a point $g_n$ which satisfies $d(g_n,x_n)\leq M$.
In particular, $d(1,g_n)$ is between $(l-\epsilon)n-M$ and $(l+\epsilon)+M$.
There is also a point on this geodesic that enters the bottleneck $U_n'$ at a point $g_n'$, satisfying $d(g_n',x_n')\leq M$ and so
we also have that $d(1,g'_n)$ is between $(l-\epsilon)n-M$ and $(l+\epsilon)+M$.
Thus, $d(g_n,g_n')\leq 2M+2\epsilon n$ and since $d(\omega_n',x_n')\leq \epsilon n$, we get
$d(\omega_n',x_n)\leq 4M+3\epsilon n$ as required.
On $\Omega_{\epsilon}$, we have
$\mu^{*n}(\omega_n')\leq \exp(-n(h-\epsilon))$ for any $n>N_{\epsilon}$ and $\mu^{*n}$ is the law of $\omega_n'$, so
$$\mathbb P(\mathbf{x}'\in \Omega_{\epsilon}, \omega'_{n}\in  B(x_n, 4M+3 n \epsilon)) \leq \sharp B(1, 4M+3n\epsilon) \exp(-n (h-\epsilon))$$
and since balls grow at most exponentially, there exists $v$ such that
$$\mathbb P(\mathbf{x}'\in \Omega_{\epsilon}, \omega'_{n}\in  B(x_n, 4M+3 n \epsilon)) \leq \exp(4vM) \exp(3n\epsilon v) \exp(-n (h-\epsilon)).$$
Hence, for any $N$ and any fixed $\mathbf{x}$ in $\Omega_{\epsilon}^N$, we have
\begin{equation}\label{equationlower-boundvisual}
\liminf_{n\to\infty} \frac{\log \mathbb{P}(\Omega_{\epsilon}\cap \mho(x_n, M))}{-n} \ge h-\epsilon  -3v \epsilon
\end{equation}
We can conclude exactly like in the proof of Proposition~\ref{lower-boundharmonicFloyd}, replacing Lemma~\ref{Lemma2.2Tanaka} by Lemma~\ref{Lemma2.2Tanakaends} and replacing Lemma~\ref{lemma3.16PY} by the following result, which asserts that partial shadows are sandwiched by balls.
\end{proof}

As before, we define the \textit{big shadow} $\Pi(g,M)$ as the set of $\xi\in \partial_{\mathcal{E}}G$ such that there is a geodesic between 1 and $\xi$ which intersects the ball of radius $M$ centered at $g$.
By definition, $\mho(g,M)\subset \Pi(g,M)$ for $M>0$.

\begin{lem}\label{lempartialshadowandvisualmetric}
Let $\xi\in \partial_{\mathcal{E}}G$ and $r\geq 0$.
Let $g$ be any point on a geodesic between 1 and $\xi$.
Then for any $M\geq 0$ there exist $C_1,C_2>0$  such that
$$\Pi(g,M) \subset B_{\rho_\lambda}(\xi, C_1r)$$
and if, in addition, $g$ lies in a $M$-bottleneck between $1$ and $\xi$, then
$$B_{\rho_\lambda}(\xi, C_2r)\subset \mho(g,M),$$
where $r=\lambda^{d(1,g)}$.
\end{lem}

\begin{proof}
Let $\zeta\in \Pi(g,M)$.
Denote by $h$ a point on  a geodesic $\alpha$ from 1 to $\zeta$ such that $d(g,h)\leq M$.
Also denote by $\beta$ the geodesic from 1 to $\xi$ in the statement of the lemma.
Consider the path $\gamma$ connecting $\zeta$ to $\xi$ obtained by following $\alpha$ from $\zeta$ to $h$, then connecting $h$ to $g$ by a geodesic and following $\beta$ from $g$ to $\xi$.
By construction, $d(1, \alpha)\ge d(1,g)-M$.
Hence, $\xi$ and $\zeta$ must lie in the same component of $\mathcal C(B_n)$ where  $n=d(1, g)-M$.
Therefore, $\rho_\lambda(\xi, \zeta)\le \lambda^{n+1}\le C_1 r$ where $C_1=\lambda^{-M-1}$ and so $\zeta\in B_{\rho_\lambda}(\xi,C_1r)$.

To prove the second inclusion, set $C_2=\lambda^{M+1}$ and consider $\zeta\in B_{\rho_\lambda}(\xi, C_2r)$.
Write $\rho_\lambda(\xi,\zeta)=\lambda^n$, where $n$ is the minimal integer such that $\xi$ and $\zeta$ lie in different components of $\mathcal C(B_n)$, so that $n\geq d(1,g)+M+1$.
Then, there exists a path $\alpha$ from $\xi$ to $\zeta$ that enters $B(1,n)$ but not $B(1, n-1)$.
In particular, we have $d(g,\alpha)\geq d(1,\alpha)-d(1,g)>M$.
Let $\beta$ be any path from 1 to $\zeta$ and consider the path $\gamma$ obtained by concatenating $\alpha$ and $\beta$,
which needs to cross $B(g,M)$, since $g$ is in a bottleneck between 1 and $\xi$.
By construction of $\alpha$, $\beta$ necessarily intersects $B(g,M)$, so $\zeta \in \mho(g,M)$. 
\end{proof}

Finally, note that Proposition~\ref{lower-boundharmonicends} yields
$$Hdim_{\rho_\lambda}(\nu_\mathcal{E})\geq \frac{-1}{\log \lambda} \frac{h}{l}$$
and Proposition~\ref{upper-boundharmonicends} yields
$$Hdim_{\rho_\lambda}(\nu_\mathcal{E})\leq \overline{Hdim}_{\rho_\lambda}(\nu_\mathcal{E})\leq \frac{-1}{\log \lambda} \frac{h}{l}.$$
Hence, $Hdim_{\rho_\lambda}(\nu_\mathcal{E})=\frac{-1}{\log \lambda} \frac{h}{l}$.
Moreover, combining these two propositions implies that for $\nu_{\mathcal{E}}$-almost every $\xi$,
$$\frac{\log \nu_{\mathcal{E}}(B(\xi,r))}{\log r}\underset{r\to 0}{\longrightarrow}Hdim_{\rho_\lambda}(\nu_{\mathcal{E}})$$
and so $\nu_{\mathcal{E}}$ is exact dimensional.
This concludes the proof of Theorem~\ref{theoremHausdorffharmonicends}. \qed

\section{Dimension of the end boundary}\label{HDimEndsSection}
Recall that  $G$ is a finitely generated group with infinitely many ends, and $(\Gamma, d)$ is the Cayley graph with respect to a finite generating set.
In this section, we compute the Hausdorff dimension of the end boundary $\partial_{\mathcal{E}}G$ endowed with a visual metric.
Define the volume growth of a subgroup $H<G$ as
$$v_H =\limsup_{n\to\infty}\frac{\log \sharp \{g\in H: d(o, g)\le n\}}{n}.$$
We prove the following.
\begin{theorem}\label{Hausdorffdimensionsetofends}
Let $G$ be a finitely generated group with infinitely many ends.
Then, for every $\lambda\in (0,1)$, we have
$$Hdim_{\rho_\lambda}(\partial_{\mathcal{E}}G)=-\frac{v_G}{\log \lambda}.$$
\end{theorem}

By Lemma~\ref{FloydDomVisu}, we have $Hdim(\partial_{\mathcal{E}}G)\leq Hdim(\partial_{\mathcal{F}}G)$
and \cite[Lemma~4.1]{PY} shows that for any $\lambda\in(0,1)$, we have $Hdim(\partial_{\mathcal{F}}G)\leq -\frac{v_G}{\log \lambda}$.
The remainder of the section is devoted to proving the following proposition, which will thus conclude the proof of Theorem~\ref{Hausdorffdimensionsetofends}.

\begin{prop}\label{LbdHdimEndsProp}
With the same notations, we have
$$Hdim_{\rho_\lambda}(\partial_{\mathcal{E}}G)\geq-\frac{v_G}{\log \lambda}.$$
\end{prop}

\subsection{Preparatory lemmas}\label{Sectionpreparatorylemmas}
Recall that by \cite[Lemma~5.1]{Bowditchconvergenceonends}, the action of $G$ on the end boundary is a convergence action.
Hence, we can consider conical points in $\partial_{\mathcal{E}}G$ and hyperbolic elements in the sense of \cite{Bowditchconvergenceonends} corresponding to this action.
Recall that an element is elliptic if it has finite order and that an infinite order element is parabolic if it fixes exactly one point on $\partial_{\mathcal{E}}G$ and hyperbolic if it fixes exactly two points on $\partial_{\mathcal{E}}G$.
Moreover, if $g$ is parabolic and fixes $\xi$, then for any $\zeta$, $g^n\cdot \zeta$ converges to $\xi$ as $n\to \pm \infty$, see \cite[Section~2]{Bow98} for more details.
It follows from Lemma~\ref{FloydDomVisu} that if $g$ is parabolic for the action on the Floyd boundary, then it is parabolic for the action on the end boundary.
Thus, an element $g\in G$ which is hyperbolic for the action on the end boundary is also hyperbolic for the action on the Floyd boundary
and by \cite[Lemma~7.2]{Yanggrowthtightness}, such an element is contracting.

\begin{lem}\label{SepHypElemLem}
Let $f$ be a hyperbolic element with two fixed points $\xi_-\ne \xi_+\in \partial_{\mathcal{E}}G$. Then there exists a finite set $E_f\in {\mathcal F} (\Gamma)$ of edges such that for every large enough $n_0>0$ and for every $n\geq 2n_0$, the two elements $1$ and $f^n$ are separated by $f^{n_0}E_f$.
Moreover, $d(1, f^{n_0}E_f)\geq d(1, f^{n_0})-c_0$ for some uniform constant $c_0$.
\end{lem}

\begin{proof}
The two distinct ends  $\xi_-\ne \xi_+$ are separated by a finite set $E_f\in \mathcal F(\Gamma)$, and since they are fixed by any power of $f$, $\xi_-\ne \xi_+$ are separated by $E_n: =f^nE_f$ for any $n$.
Note that $n\in \mathbb Z \mapsto f^n$ is a quasi-geodesic so that the two half-rays converge to the corresponding ends $\xi_-$ and $\xi_+$.
Thus, whenever $n_0$ is large enough, $f^{-n_0}$ and $f^{n-n_0}$ are separated by $E_f$ for any $n\geq 2n_0$.
Up to translation, 1 and $f^n$ are thus separated by $E_{n_0}$.
Moreover, $d(1, f^{n_0})\leq d(1, f^{n_0}E_f)+\sup_{x\in E_f} d(f^{n_0}x,f^{n_0})$ and since $E_f$ is fixed, the second term in the right-hand side is uniformly bounded, which concludes the proof.
\end{proof}

Let $F$ be a set of three pairwise independent hyperbolic elements in $G$.
We write $F^n=\{f^n: f\in F\}$ for any $n\ge 1$.
\begin{lem}\label{ExtensionLem}
There exists an integer $n_1>0$ with the following property for every $n\ge n_1$.
For any $g, h\in G$, there exist $f\in F$ and a finite set $E_f$ of edges separating $1$ and $f^n$ and such that any path between $g^{-1}$ and $f^nh$ has to cross $E_f$.
\end{lem}

\begin{proof}
Since hyperbolic elements are contracting, by \cite[Lemma~2.14]{YangSCC}, there exist $n_1, \epsilon>0$  with the following property:
for any $n>n_1$ and $f^n\in F^{n}$ the points  $1$ and $f^n$ stay within the $\epsilon$-neighborhood of any geodesic $[g^{-1}, f^nh]$, so that the path $\gamma=[g^{-1},1][1,f^n][f^n,f^nh]$ is a $(1, 4\epsilon)$-quasi-geodesic. 
Let $E$ be provided by Lemma~\ref{SepHypElemLem} for the element $f$.
Hence, for large enough $n_0$, assuming that $n\geq 2n_0$, the two elements $1$ and $f^n$ are separated by $E_f:=f^{n_0}E$.

For given $\lambda, c$, we choose $n_0$ to satisfy  $d(1,F^{n_0})> c+\lambda \mathrm{diam}(E_f)+c_0$.
\begin{claim}\label{SeparationClaim}
If $\gamma$ is a $(\lambda,c)$-quasi-geodesic path containing $[1,f^n]$ for some $\lambda, c>0$, then  the endpoints of $\gamma$ are separated by $E_f$.
\end{claim}
\begin{proof}[Proof of the Claim]
Denote by $\gamma_1$ and $\gamma_2$ the corresponding subpaths before and after $[1,f^n]$. Observe that $\gamma_1$ and $\gamma_2$ are disjoint with $E_f$. Indeed, assume by contradiction that $x\in \gamma_1\cap E\ne\emptyset$, the case for $\gamma_2$ being symmetric.
By Lemma~\ref{SepHypElemLem}, the elements $1, f^n$ are separated by $E_f$ and $d(1, E_f)\geq d(1,f^{n_0})-c_0$.
We then choose $y\in [1,f]\cap E_f$ so that $d(1, y) \ge d(1, f^{n_0})-c_0$.
This yields a subpath $\alpha$ of $\gamma$ whose endpoints $x$ and $y$ are at most $\mathrm{diam}(E_f)$ apart.
Since $x,1,y$ are aligned in this order on $\alpha$, the quasi-geodesicity implies that $d(1, y)\le \lambda d(x, y)+ c$. This  contradicts the choice of $n_0$ above. 

As a consequence, any path  between the two endpoints of $\gamma$ pass through $E_f$.
Indeed, if there was such a path  disjoint with $E_f$, we would obtain a path from $1$ to $f^n$ disjoint with $E_f$, contradicting Lemma~\ref{SepHypElemLem}. 
\end{proof} 

The proof is concluded by the Claim applied to the $(1, 4\epsilon)$-quasi-geodesic $\gamma=[g^{-1},1][1,f^n][f^n,f^nh]$.
\end{proof}

Let $r=\max \{\mathrm{diam}(E_f): f\in F\}$, where $E_f$ is given by Lemma~\ref{ExtensionLem}.
The following definition refines the notion of bottleneck given in the last section.

\begin{definition}
Let $x,y\in G\cup \partial_\mathcal{E}G$ and let $n_1$ be given by Lemma~\ref{ExtensionLem}.
A \textit{$(r, F)$-bottleneck point}  between $x, y$ is a point $b\in G$ such that any path between $x$ and $y$ has to cross the $r$-neighborhood of $b[1, f]$ for some $f\in F^{n_1}$.
\end{definition}

Note that if $b$ is a $(r, F)$-bottleneck point between $x$ and $y$, then $b$ is in $b[1, f]$ so it lies in a $M$-bottleneck for $M=(d(1,F^{n_1})+2r)$ in the sense of Definition~\ref{defbottleneck1}.
An immediate result  follows by the same argument in Lemma \ref{FloydDomVisu}.
\begin{lem}\label{BNisFloydLargeLem}
If $b$ is a $(r, F)$-bottleneck point between $g,h$, then the Floyd distance $\delta_{\lambda,b}(g, h)$ based at $b$ is bounded below by a constant depending $\lambda, r, F$ only.
\end{lem}

\subsection{Construction of rooted geodesic trees with bottlenecking property}
The strategy in proving Proposition~\ref{LbdHdimEndsProp} is similar to \cite{PY}.
We shall construct a sequence of  rooted quasi-geodesic trees so that the Hausdorff dimension of their ends  tends to the Hausdorff dimension of $\partial_{\mathcal E} G$.  

Before getting into the construction, we introduce the following definition of ends with the uniform bottlenecking property, which are uniform conical points in the sense of \cite{PY}. 
{We fix $r$, $F$ and $n_1$ as in Section~\ref{Sectionpreparatorylemmas}.}
\begin{definition}
A path $\gamma$ in $\Gamma$ has the \textit{$L$-bottlenecking property} for some $L>0$ if there exists a sequence of $(r, F^{n_1})$-bottleneck points $b_i\in \gamma$ between the endpoints of $\gamma$ such that $\sup\{d(b_i, b_{i+1}): i\ge 1\}\le L$ and $d(x,\{ b_{i+1}: i\ge 1\})\le L/2$ for any $x\in \gamma$.

An  end $\xi$ has the \textit{$L$-bottlenecking property} for some $L>0$ if there exists a geodesic ray $[1, \xi]$ with the $L$-bottlenecking property.
\end{definition}

\begin{remark}
By definition, an  end   with   $L$-bottlenecking property must be thin, but the converse is false. 
\end{remark}
 
If $H$ is a free semigroup with a free basis $B$, then the standard Cayley graph $T$ of $H$ with respect to $B$  is a rooted tree at $1$, where every edge with unit length is labeled by a letter in $B$.

The construction of rooted trees in the proof of Proposition~\ref{LbdHdimEndsProp} will be given by the standard Cayley graphs of a sequence of free semigroups described in the following lemma. Given $n, \Delta>0,$  define the annulus set
$$
A(n, \Delta) =\{g\in G: |d(1, g)-n|\le \Delta\}.
$$
 \begin{lem}\label{freesemigroup}
For any   $0<v<v_G$, there exists a free semigroup $H$ with a free basis $B$ and a constant $L=L(v)>0$ with $\lim_{v\to v_G} L(v)=\infty$ so that the following hold.
\begin{enumerate}
\item
There exists a quasi-isometric embedding map $\Phi$ from the standard Cayley graph $T$ of $H$ into the Cayley graph $\Gamma$ of $G$, which is induced  by the inclusion $H\subset  G$ and such that each geodesic issuing from the identity  in $T$ is sent to a path with the $L$-bottlenecking property.
\item
The map $\Phi$ extends to a topological embedding $\partial \Phi$ of the end boundary $\partial_{\mathcal E} H$ of $H$ into $\partial_{\mathcal E}G$.
\item
We have $v_H>v$ and $C_1\exp(Lv_H) \le\sharp H\cap A(1, n, \Delta) \le C_2\exp(Lv_H)$ for $n\ge 1$, where $C_1$ and $C_2$ depend on $\Delta$.
\end{enumerate}
\end{lem}
\begin{proof} 

(1)
We set $L_\pm=n\pm(\Delta+d(1,F^{n_1}))$.
Let  $\mathbb W(A)$ be the set of   all finite words over an alphabet set  $A$.  Given a set $A \subset \Gamma$,   we can define an \textit{extension map} $\Phi: \mathbb W(A) \to G$ as follows: for any word $W=a_1a_2\cdots a_n \in \mathbb W(A)$, there exists a sequence $f_i\in F^{n_1}$ such that  the path $\gamma$ labeled by 
\begin{equation}\label{AdmisEQ}
\Phi(W)=  a_1  \cdot f_1\cdot  a_2 \cdot f_2\cdot \cdots \cdot   a_{n-1} \cdot f_{n-1} \cdot  a_n \cdot f_n  \in G,
\end{equation}
is a $(\lambda, c)$-quasi-geodesic for fixed constants $\lambda, c>0$ depending only on $F$.
Moreover, the path $\gamma$ labeled by $\Phi(W)$ has the $L_+$-bottlenecking property.
The choice of the points $f_i$ is made by iterating the construction given by Lemma~\ref{ExtensionLem}.
Precisely, by \cite[Lemma 2.16]{YangSCC}, one can choose the $f_i$ so that $\gamma$ is a $(\lambda, c)$-quasi-geodesic with fixed $\lambda$ and $c$.
Then, by the Claim \ref{SeparationClaim},  the terminal endpoints of the segments labeled by $a_i$ ($1\le i\le n$) give $n$ distinct $(r, F)$-bottleneck points between $1$ and $\Phi(W)$ with pairwise distance between $L_-$ and $L_+$.
Those bottlenecks shall be called  \textit{canonical}.

To generate a free semigroup, we need the following property, proved in \cite[Lemma 2.19]{YangSCC}. There exists  a  fixed constant  $C_0>0$ such that for every large enough $n$ and every $C\ge C_0$, there exists a $C$-separated set $A{\subset }A(n, \Delta)$ with the following properties: \begin{enumerate}[(a)]
    \item
    $\sharp A \succ_{C} \sharp A(n, \Delta)$,
    \item 
    there is a common $f\in F^{n_1}$  for  each pair $(a, a') \in A\times A$ in   the   path (\ref{AdmisEQ}) with the $L$-bottlenecking property.
\end{enumerate}

If $C$ is taken sufficiently large, then the  map $\Phi: (Af)^n\to G$ as above is injective. Indeed, consider two paths  $\beta_1=\Phi(W_1)$ and $\beta_1'=\phi(W_1')$, where $W_1=a_1f\cdots a_mf$ and $W_1'=a_1'f\cdots a_n'f$. We shall prove the following stronger fact, used later on in Lemma \ref{cpvismetricLem}.
Denote by $(\beta_1)_+$, respectively $(\beta_1')_+$ the terminal endpoint of $\beta_1$, respectively $\beta_1'$.
\begin{claim}\label{shortFellowTravelClaim}
If $d((\beta_1)_+,(\beta_1')_+)\le D$, then $d(a_1, a_1') < C$ for some $C=C(D)$.
\end{claim}
By choosing $C\geq C(0)$ the claim proves the injectivity of $\Phi$. Hence,  the set $B:=Af$ generates a free semi-group denoted by $H$.

\begin{proof}[Proof of the Claim]
Write $g=\Phi(W_1)=(\beta_1)_+$ and $g'=\Phi(W_1')=(\beta_1')_+$.
Since  the point $a_1$ in $\beta_1$   is a $(r, F)$-bottleneck point between $1$ and $g$, by Lemma~\ref{BNisFloydLargeLem} and Lemma~\ref{karlssonlem} applied to the Floyd distance $\delta_{\lambda, a_1}$ based at $a_1$, there exists a constant $C$ depending only on $r, F$ such that $d(a_1, [1,g])< C$.
Since $d(g,g')\leq D$, concatenating the geodesic $[1,g]$ with a geodesic from $g$ to $g'$ yields a quasi-geodesic $\alpha$ whose parameters only depend on $D$.
Hence, Lemma~\ref{karlssonlem} shows that $d(a_1',\alpha)<C'$ where $C'$ only depends on $r,F$ and $D$.
Using again that $d(g,g')\leq D$, we finally obtain that $d(a_1',[1,g])\leq C''$. By the choice of $a_1, a_1'\in A \subset A(n, \Delta)$, we have $|d(1,a_1)-d(1,a_1')|\le 2\Delta$.  We deduce  that $d(a_1, a_1')< C+C''+2\Delta$, which concludes the proof.
\end{proof}

(2) Note that the image of each geodesic ray $\gamma$ in $T$ is a path with the bottlenecking property which converges to an end in $\partial_{\mathcal E}G$. We deduce that the map $\gamma\mapsto \phi(\gamma)$ induces the desired topological  embedding.    
 
(3) For a given $v$, one can choose $n$ big enough such that $v_H>v$ for a free semi-group $H=\langle Af\rangle$ constructed as above.
This is possible since $\sharp A \succ_{C} \sharp A(n, \Delta)$, and $$v=\lim_{n\to \infty}\frac{\log \sharp A(n,\Delta)}{n}.$$ 
See \cite[Section 3]{YangSCC} for the details. 
The purely exponential growth
$$\sharp H\cap N(o, r)\asymp \exp(v_H r)$$
follows by standard arguments using the fact that the subgroup $H$ is  contracting, as stated there. See an argument in the proof of \cite[Lemma 3.9]{PY}.  
\end{proof}

\subsection{Completion of proof of Proposition \ref{LbdHdimEndsProp}}
We shall first prove that the embedding $\partial \Phi: \partial_{\mathcal E} H\to \partial_{\mathcal E}G$ is bi-Lipschitz with respect to visual metrics for appropriate choice of parameters. The Hausdorff dimension of $\partial_{\mathcal E}G$ will then be bounded from below by that of the ends of the rooted tree $T$ which will be easier to compute. 

Recall  $T$ is the standard Cayley graph of $H$ with respect to the free base $B=Af$ whose edges with unit length are labeled by a letter in $B$.  

\begin{lem}\label{cpvismetricLem}
Under the assumption of Lemma \ref{freesemigroup}, assume that $\Phi(p)=\xi$ and $\Phi(q)=\eta$ where $p, q\in \partial_{\mathcal E} T$.  Let $m$ be the length of the intersection $[1,p]\cap[1,q]$ where $[1, p], [1,q]$ are geodesic rays in $T$. Then 
$$
\rho_\lambda(\xi, \eta) \asymp_L \lambda^{mL}.
$$
\end{lem}
\begin{proof}
In the proof, let us write $$[1,p]=(a_1f)(a_2f)\cdots (a_mf)(a_{m+1}f)\cdots$$ and $$[1,p]=(a'_1f)(a'_2f)\cdots (a'_mf)(a'_{m+1}f)\cdots$$ and denote  $\beta_1=\Phi([1,p])$ and $\beta_2=\Phi([1,q])$.
Those are two rays with the $L$-bottlenecking property ending at $\xi$ and $\eta$.
Let   $[1,p]\cap[1,q]=s_1\cdots s_m$ be a geodesic of length $m$ where $s_i\in Af$.
In other words, $a_i=a_i'$ for $1\le i\le m$.

By definition, $\rho_\lambda(\xi, \eta)=\lambda^n$ where  $n$ is the minimal integer such that $\xi$ and $\eta$ belongs to different components of the complement of the ball $B_n$.
Note  that $n\ge mL/2$.
Indeed, the path $\tau=\Phi(s_1\cdots s_m)$ issuing from $1$ has the $L$-bottlenecking property, ends at $u=s_1\cdots s_m$ and is contained in both $\beta_1$ and $\beta_2$. Thus, there are $m$ bottleneck points between   $1$ and $u$ with pairwise distance at least $L/2$, so that the geodesic $[1, u]$ in the Cayley graph $\Gamma$ of $G$  has to pass through them in order.
We then derive that  $d(1, u)\ge mL/2$.
Since the path obtained from $\beta_1\cup\beta_2\setminus \tau$ connects $\xi$ to $\eta$,       the definition of $n$ implies  that $n>mL/2$.  

We shall prove the upper-bound $n\le R:=(m+1)L+r$, which will conclude the proof. By definition, it suffices to show  that $\xi$ and $\eta$ belongs to the distinct components of $G\setminus B_R$.

Let $b$   be the second     $(r, F)$-bottleneck point on $\beta_1$ after the initial subpath $\tau$ of $\beta_1$: it is the terminal endpoint of $a_{m+2}$ in $\beta_1$. Thus, any path from $1$ to $\xi$ has to cross the $r$-neighborhood of $b[1,f]$.  
Let $C=C(r+d(1,F^{n_1}))$ be the constant   given by the Claim~\ref{shortFellowTravelClaim}.
If $A$ is chosen to be $C$-separated, then  $d(b, \beta_2)>r+d(1,F^{n_1})$  and thus,  $\beta_2\cap N_r(b[1,f]) =\emptyset$.  
 
It remains to prove that any path from  $\xi$ to $\eta$ has to intersect the ball $B_R$. If not, let $Q$ be a path from     $\xi$ to $\eta$ so that $Q\cap B_R=\emptyset$. Recall that $b$ is the $(m+1)$-th $(r, F)$-bottleneck point and thus $d(1, bf)\le (m+1)L$. The value of $R$ implies that  $B_R$ contains $N_r(b[1,f])$.
Also recall that $\beta_2\cap N_r(b[1,f]) =\emptyset$, so concatenating $Q$ and $\beta_2$ yields a path $Q\cdot \beta_2$ from $\xi$ to $1$   avoiding the set $N_r(b[1,f])$. This is a contradiction, since $b$ is a $(r, F)$-bottleneck point between $1$ and $\xi$.
As desired,  $\xi$ and $\eta$ lie in different   components in the complement of $B_R$. The proof of the upper bound $n\le R$ follows.
\end{proof}

\begin{prop}\label{BiLipembedPop}
Let $\alpha=\lambda^L$ for given $\lambda\in (0,1)$.  Then the   embedding   $\partial \Phi$   from the ends boundary $(\partial_{\mathcal E} T, \rho_\alpha)$  into the ends boundary $(\partial_{\mathcal E}G, \rho_\lambda)$  is bi-Lipschitz.
\end{prop}

\begin{proof}
By Lemma \ref{cpvismetricLem}, we see $\rho_\alpha(\xi, \eta) = \alpha^{m} = \lambda^{mL} \asymp_L   \rho_\lambda(\Phi(p),\Phi(q))$.  Thus, the map from $\partial_{\mathcal E} T$   is bi-Lipschitz onto its image in $\partial_{\mathcal E} G$.  
\end{proof}
 
Let $T$ be any infinite rooted tree, with vertex set $V$ partitioned by depth (distance from the root vertex): $V = \cup_{n=0} V_n$. Let $N_n$ and $M_n$ be increasing sequences of positive integers. Say that $T$ is \textit{$\{M_n\}$-regular relative to $\{N_n\}$} if for every $n \ge 1$ every vertex $x$ at depth $N_n$ has exactly $M_{n+1}$ descendant vertices at depth $N_{n+1}$.  
\begin{lem}\cite[Lemma 2]{LS00}\label{RTreeHdimLem}
Let $T$ be an infinite rooted tree with space of ends $\partial_{\mathcal E} T$ equipped with the visual metric of parameter $\alpha\in (0,1)$. If there is an increasing sequence of integers $N_n$ satisfying $\lim_{n\to\infty} \frac{ N_{n+1}}{N_n} = 1$ such that, for
some sequence $\{M_n\}$ of positive integers, $T$ is $\{M_n\}$-regular with respect to $\{N_n\}$, then the Hausdorff dimension of $\partial_{\mathcal E} T$ (relative to the visual metric $\rho_\alpha$) is given by
$$
Hdim_{\rho_\alpha} (\partial_{\mathcal E} T ) = \liminf_{n\to\infty} \frac{\log |V_{N_n}|}{-N_n\log \alpha}
= \liminf_{n\to\infty}\frac{\log  \prod_{j=1}^n M_j}{-N_n\log \alpha}.
$$
\end{lem}
 
We apply  Lemma \ref{RTreeHdimLem} to the standard Cayley graph $T$ of $H$, where $M_n=\sharp Af$ and $N_n=n$.  We obtain $$Hdim_{\rho_\alpha} (\partial_{\mathcal E} T ) = \liminf_{n\to\infty}\frac{\log \prod_{j=1}^n \sharp Af}{-n\log \alpha}=\frac{\log \sharp Af}{-L\log \lambda}.$$

By  Lemma \ref{freesemigroup}, $C_1\exp(Lv_H) \le \sharp Af \le C_2\exp(Lv_H)$, so
$$Hdim_{\rho_\alpha} (\partial T )\geq -\frac{\log C_2 + Lv_H}{L\log \lambda},$$
where $v<v_H\leq v_G$.

Since a bi-Lipschitz map preserves the Hausdorff dimension, by Proposition \ref{BiLipembedPop}, we have 
$$Hdim_{\rho_\lambda} (\partial_{\mathcal{E}}\Gamma)\ge -\frac{\log C_2 + Lv}{L\log \lambda}.$$
As $v$ tends to $v_G$, $L$ goes to infinity, so we finally obtain
$$Hdim_{\rho_\lambda} (\partial_{\mathcal{E}}\Gamma)\ge -\frac{v_G}{\log \lambda}.$$
The proof of Proposition \ref{LbdHdimEndsProp} and thus Theorem \ref{Hausdorffdimensionsetofends} is complete. \qed

\section{Characterizing the Doubling property: proof of Theorem \ref{CharDoublingThm}}\label{Sectiondoubling}
It is well-known that a virtually free group is hyperbolic and that its end boundary endowed with a visual metric is bi-H\"older to its Gromov boundary endowed with a Gromov's visual metric.
The latter has the doubling property, since the Patterson-Sullivan measure is doubling (even Ahlfors regular) by the work of Coorneart \cite{Coornaert}.   
We prove here the following result.
\begin{prop}\label{propinaccessible}
If a finitely generated group admits a splitting over finite edge groups as a finite graph of groups with at least one one-ended vertex group, then the visual metric is not doubling.
\end{prop}
Theorem~\ref{CharDoublingThm} is then a consequence of this proposition.
Indeed, accessible groups admit a splitting over finite edge groups as a finite graph of groups $\mathcal{G}$, so that the vertex groups either are finite or one-ended.
In the former case, the group is virtually free, in the later case, it satisfies the assumptions of Proposition~\ref{propinaccessible}.

\begin{proof}
Consider a splitting over finite edge groups as a finite graph of groups $\mathcal{G}$ and let $H$ be a one-ended vertex group.
Notice that a metric space $(X, d)$ is   {doubling} if and only if for any (or some) $\theta\in (0,1)$, there exists $N=N(\theta)>0$ such that  every ball of radius $s>0$ can be covered by at most $N$ balls of radius $\theta s$.
To prove that the end boundary is not doubling, our strategy will be as follows.
We will consider for every $n$ the ball of radius $s=\lambda^n$ centered at the end $\xi$ of $G$ corresponding to the unique end of $H$.
We will prove that for some fixed $\theta$, this ball cannot be covered by $N(n)$ balls of radius $\theta s(n)$, where $N(n)$ goes to infinity.
This will conclude the proof.
We present the details below.

{By the Bass-Serre theory,
$G$ is isomorphic to either an amalgamated product $H*_F K$ or a HNN extension $H*_{F}$ over a finite group $F$, where $K$ might not be one-ended. 
In what follows, we only consider the amalgamated product case, the HNN extension case being similar.
Let $G=H*_F K$ so that $H$ and $K$ are generated by two finite sets $S$ and $T$ respectively and assume for simplicity that both sets $S$ and $T$ contain $F$. 
The normal form given by \cite[(9.2),(9.3)]{Woessends} shows that the Cayley graph $\Gamma$ of $G$ with respect to $S\cup T$ is obtained as the disjoint union
$$
\displaystyle \bigsqcup_{g\in G} \{g\cdot Cay(H, S), g\cdot Cay(K,T)\}.
$$
glued along cosets $gF$ between $g\cdot Cay(H, S), g\cdot Cay(K,T)$.
In particular, for any end $\eta\neq \xi$, any path from $\eta$ to $\xi$ has to pass through a finite set $F_\eta$ corresponding to some $F$-coset in $H$ and any path from $1$ to $\eta$ also has to pass through $F_\eta$.
Since we made the assumption that $F\subset S\cap T$, the diameter of $F_\eta$ is at most 1, that is any two distinct vertices in $F_\eta$ are connected by an edge.}

Let $S_n:=\{g\in G: d(1, g)=n\}$ be the $n$-sphere in the Cayley graph $\Gamma$ of $G$. By definition of the visual metric, if $\rho_\lambda(\xi, \eta)=\lambda^n$ for $n\ge 1$, then every path  from $\xi$ and $\eta$ stays within distance $n$ of the identity and  exits $H$ at some vertex in $F_\eta$. Moreover there exists such a path  from $\xi$ and $\eta$ which is disjoint from $B_{n-1}$.
Recalling that the diameter of $F_\eta$ is at most 1, we see that $d(1, F_\eta)\ge n-1$.
Assume by contradiction that $d(1,F_\eta)>n$. Then, the sub-path from $\eta$ to $F_\eta$ stays in the complement of the ball $B_n$ of radius $n$.
Indeed, if this were not the case, this sub-path would pass through a point $g\in B(1,n)$ and so concatenating a path from $1$ to $g$ and the path from $g$ to $\eta$ would yield a path from $1$ to $\eta$ not passing through $F_\eta$.
Now, since $\mathrm{Cay}(H,S)$ is one-ended, one can join $F_\eta$ with $\xi$ by a path which also stays in the complement of $B_n$, contradicting the fact that any path from $\eta$ to $\xi$ needs to pass through this ball.
Thus, $d(1,F_\eta)\leq n$.
Similarly, if $ d(1, F_\eta)\ge n$, then $\rho_\lambda(\xi, \eta)\le \lambda^n$.  Set $s=\lambda^n$.
We can reformulate the above discussion by the following inclusions :
\begin{equation}\label{equationdoublingproperty}
\{\eta\neq \xi\in \partial_{\mathcal E}\Gamma: d(1, F_\eta)\ge n\}\subseteq B^{\rho_\lambda}(\xi, s)\setminus \xi \subseteq \{\eta\neq \xi\in \partial_{\mathcal E}\Gamma: d(1, F_\eta)\ge n-1\}.
\end{equation}

Since $H$ is one-ended, it is infinite and it is not virtually cyclic. By Gromov's polynomial growth theorem, the growth function of $H$ is at least quadratic, hence super-linear.
Therefore, fixing $k\ge 2$, the number of elements in $$S_{n+k}(H):= S_{n+k}\cap H$$ grows at least linearly in $n$ and in particular goes to infinity as $n$ goes to infinity.

For any   point $h\in S_{n+k}(H)$, choose an end $\eta=\eta(h)\ne \xi \in \partial_{\mathcal E}\Gamma$ so that $h$ lies in $F_\eta$. By definition of $h$, $d(1,F_\eta)\leq n+k$ and since the diameter of $F_\eta$ is at most 1, we also have $d(1, F_\eta)\ge n+k-1.$
According to~(\ref{equationdoublingproperty}), $\rho_\lambda(\xi, \eta)\le \lambda^{n+k-1}$.
Thus, $\eta\in B^{\rho_\lambda}(\xi, s)$.

Observe  that for any two elements $h_1\ne h_2\in S_{n+k}(H)$ with $d(h_1, h_2)\ge 2$, we claim that
$$
\rho_\lambda(\eta_1,\eta_2) \ge \lambda^{n+k+1}
$$
where $\eta_1=\eta({h_1}),\eta_2=\eta({h_2})$.
Indeed, $d(h_1, h_2)\ge 2$ implies that $F_{\eta_1}\cap F_{\eta_2}=\emptyset$. Notice that any path from $\eta_1$ to $\eta_2$ has to pass through both $F_{\eta_1}$ and $F_{\eta_2}$, otherwise one would produce a path from $1$ to $\eta_i$ not passing through $F_{\eta_i}$ for some $i\in \{1,2\}$.
Hence, any path from $\eta_1$ to $\eta_2$ has to intersect the ball $B_{n+k+1}$ and the claim follows.

We are ready to finish the proof. 
Set $\theta=\lambda^{k+1}$ and let $N(n)$ be the maximal size of a set $\Sigma\subset S_{n+k}(H)$ such that any two elements $h,h'$ in $\Sigma$ satisfy that $d(h,h')\geq 2$.
Then, $N(n)$ goes to infinity as $n$ goes to infinity.
We can thus produce $N(n)$ points which are $\theta s$-separated  in $B_{\rho_\lambda}(\xi, s)$.
In other words, the ball $B_{\rho_\lambda}(\xi, s)$ cannot be covered by $N(n)$ balls of radius $\theta s$.
Since $k$ is fixed and $N(n)$ goes to infinity as $n$ goes to infinity, this provides a contradiction with the definition of the doubling property.
This concludes the proof.
\end{proof}

\bibliographystyle{plain}
\bibliography{HD}

\end{document}